\documentclass[12pt]{article}

\usepackage{amsmath,amsthm,amsfonts,amssymb, color,graphicx}

\newtheorem{theorem}{Theorem}[section]

\newtheorem{proposition}[theorem]{Proposition}
\newtheorem{lemma}[theorem]{Lemma}

\newtheorem{remark}[theorem]{Remark}

\allowdisplaybreaks[4]

\def\bD{\mathbb D}
\def\bE{\mathbb E}
\def\bN{\mathbb N}
\def\bP{\mathbb P}
\def\bR{\mathbb R}

\def\cA{\mathcal A}
\def\cB{\mathcal B}
\def\cD{\mathcal D}
\def\cF{\mathcal F}
\def\cH{\mathcal H}
\def\cI{\mathcal I}

\def\cD{\mathcal{D}}

\def\cF{\mathcal{F}}
\def\cG{\mathcal{G}}
\def\cI{\mathcal{I}}

\def\cH{\mathcal{H}}
\def\cP{\mathcal{P}}

\def\cS{\mathcal{S}}

\def\fX{\mathfrak{X}}

\def\bD{\mathbb{D}}
\def\bE{\mathbb{E}}
\def\bR{\mathbb{R}}

\def\e{\varepsilon}

\begin{document}

\title{Central limit theorems for heat equation with time-independent noise: the regular and rough cases}

\author{Raluca M. Balan\footnote{University of Ottawa, Department of Mathematics and Statistics, STEM Building, 150 Louis-Pasteur Private,
Ottawa, ON, K1N 6N5, Canada. E-mail: rbalan@uottawa.ca. Research
supported by a grant from the Natural Sciences and Engineering
Research Council of Canada.} \and   Wangjun Yuan\footnote{Corresponding author. University of Ottawa, Department of Mathematics and Statistics, STEM Building, 150 Louis-Pasteur Private,
Ottawa, ON, K1N 6N5, Canada. E-mail: ywangjun@connect.hku.hk.}}

\date{May 25, 2022}
\maketitle

\begin{abstract}
\noindent In this article, we investigate the asymptotic behaviour of the spatial integral of the solution to the parabolic Anderson model with time independent noise in dimension $d\geq 1$, as the domain of the integral becomes large. We consider 3 cases: (a) the case when the noise has an integrable covariance function; (b) the case when the covariance of the noise is given by the Riesz kernel; (c) the case of the rough noise, i.e. fractional noise with index $H \in (\frac{1}{4},\frac{1}{2})$ in dimension $d=1$. In each case, we identify the order of magnitude of the variance of the spatial integral, we prove a quantitative central limit theorem for the normalized spatial integral by estimating its total variation distance to a standard normal distribution, and we give the corresponding functional limit result.
\end{abstract}

Primary 60H15; Secondary 60H07, 60G15, 60F05


\tableofcontents

\section{Introduction}
\label{section-intro}

In the recent years, there has been a lot of interest in investigating the asymptotic properties of the spatial average of the solution to the stochastic heat and wave equations driven by a Gaussian noise, which is white or colored in space and time. This line of investigation was initiated in the seminal article \cite{HNV20} which considered the case of the heat equation driven by a space-time white noise in dimension $d=1$, with initial condition 1 and non-linear factor $\sigma(u)$ multiplying the noise, for a Lipschitz function $\sigma$.
In \cite{HNV20}, the authors proved that for any fixed time $t>0$, the
spatial average of the solution $u$:
\[
F_R(t)=\int_{-R}^R u(t,x)dx
\]
has approximately a normal distribution, as $R\to \infty$. More precisely, these authors showed that the variance $\sigma_R^2(t):=\bE[F_R^2(t)] \sim C R$ as $R \to \infty$, and estimated that the total variation distance ($d_{TV}$) between $F_R(t)/\sigma_R(t)$ and $Z \sim N(0,1)$ is of the order $R^{-1/2}$, a result which is known in the literature as the {\em quantitative central limit theorem} (QCLT). This was achieved by using the powerful methods of Malliavin calculus, combined with Stein's method for normal approximations, an area which is explored in depth in the monograph \cite{NP}. Since then, there has been a steady flow of new contributions to this area, which includes the present article.
We review below some of the most important contributions in this area.

The case of the heat equation with white noise in time, Lipschitz function $\sigma$, and spatial covariance given by the Riesz kernel $\gamma(x)=|x|^{-\beta}$ for $\beta \in (0,d\wedge2)$ was studied in \cite{HNVZ}, where it was proved that $\sigma_R^2(t) \sim C R^{2d-\beta}$ and $d_{TV} \leq CR^{-\beta/2}$. When the noise is colored in time, the situation is more complex, and even the basic question of existence of solution remains an open problem for general functions $\sigma$, except in the linear case $\sigma(u)=u$ (known as the {\em parabolic Anderson model}).
The problem of QCLT for this model with fractional noise in time with index $H_0>1/2$ was considered in \cite{NZ20} in two cases: (a) $\gamma \in L^1(\bR^d)$ and (b) $\gamma(x)=x^{-\beta}$ for $\beta \in (0,d\wedge2)$. It was shown there that $\sigma_R^2(t) \sim C R^d$ in case (a), and $\sigma_R^2(t) \sim C R^{2d-\beta}$ in case (b). The method of \cite{NZ20} uses the multivariate chaotic central limit theorem, which yields the normal approximation, but does not give the rate of the convergence. This rate was obtained in the recent preprint \cite{NXZ}, where it was shown that
$d_{TV} \leq CR^{-d/2}$ in case (a) and $d_{TV}\leq CR^{-\beta/2}$ in case (b), using a different method based on an improved version of the second-order Gaussian Poincar\'e inequality due to \cite{vidotto}, which was used for the first time in this context in \cite{BNQZ}. In addition, \cite{NXZ} considers the more difficult case (c) of the fractional noise in space with index $H<1/2$ (in dimension $d=1$), which is also fractional in time with index $H_0>1/2$ satisfying $H_0+H>3/4$. In this case, the authors discovered the surprising rates
$\sigma_R^2(t) \sim C R$ and $d_{TV} \leq CR^{-1/2}$, which show that once we descend below the value $1/2$, the spatial index $H$ does not have any impact on the rates in the QCLT. In the present article, using similar methods to \cite{NXZ}, we consider the QCLT problem for the parabolic Anderson model with time-independent noise (which corresponds formally to the case $H_0=1)$, and spatial covariance given by cases (a),(b),(c) above.
In addition, we prove the corresponding functional limit result.

The case of the wave equation in dimension $d=1$ with white noise in time and non-linear factor $\sigma(u)$ multiplying the noise (for a Lipschitz function $\sigma$) was studied in \cite{DNZ20} in the case of fractional noise in space with index $H\geq 1/2$, where it was shown that $\sigma_R^2(t) \sim C R^{2H}$ and $d_{TV} \leq CR^{H-1}$. The extension to the case $d=2$ was done in \cite{NZ20} and \cite{BNZ21} in cases (a), respectively (b) mentioned above, assuming in addition that $\gamma\in L^{\ell}(\bR^d)$ for some $\ell>1$ in case (a). More precisely, it was shown in \cite{BNZ21} that $\sigma_R^2(t) \sim C R^{d}$ and $d_{TV} \leq CR^{-1}$ in case (a), and $\sigma_R^2(t) \sim C R^{4-\beta}$ and $d_{TV} \leq CR^{-\beta/2}$ in case (b). Similarly to the heat equation, for the wave equation with colored noise in time, the existence of solution is an open problem, except in the case $\sigma(u)=u$ (known as the {\em hyperbolic Anderson model}). The QCLT problem for this model with colored noise in time was considered in \cite{BNQZ} in dimension $d\leq 2$, in cases (a) and (b) above: in case (a),
$\sigma_R^2(t) \sim C R^{d}$ and $d_{TV} \leq CR^{-d/2}$, while in case (b), $\sigma_R^2(t) \sim C R^{2d-\beta}$ and $d_{TV} \leq CR^{-\beta/2}$. The same problem for the time-independent noise has been considered in the recent preprint \cite{BY}.

\medskip

After this brief literature review, we explain now the problem considered in the present article. We are interested in the parabolic Anderson model with time-independent noise:
\begin{align} \label{eq-heat}
	\begin{cases}
		\dfrac{\partial u}{\partial t} (t,x)
		= \dfrac{1}{2} \Delta u(t,x) + u(t,x) \dot{W}(x), \
		t>0, \ x \in \bR^d, (d \geq 1)\\
		u(0,x) = 1.
	\end{cases}
\end{align}

The noise is given by a zero-mean Gaussian process $\{W(\varphi);\varphi \in \cD(\bR^d)\}$, defined on a complete probability space $(\Omega,\cF,\bP)$, with covariance
\[
\bE[W(\varphi)W(\psi)]
=\int_{\bR^d}\cF \varphi(\xi) \overline{\cF \psi(\xi)}\mu(d\xi)
=:\langle \varphi,\psi\rangle_{\cP_0},
\]
where $\mu$ is a tempered measure on $\bR^d$, called the spectral measure of the noise. Here $\cD(\bR^d)$ is the space of infinitely differentiable functions on $\bR^d$ with compact support, and $\cF \varphi (\xi)=\int_{\bR^d}e^{-i \xi \cdot x}\varphi(x) dx$ is the Fourier transform of $\varphi$.

We let $\cP_0$ be the Hilbert space defined as completion of $\cD(\bR^d)$ with respect to the inner product $\langle \cdot, \cdot \rangle_{\cP_0}$. By the isometry property, the noise can be extended to an isonormal Gaussian process $\{W(\varphi);\varphi \in \cP_0\}$, as defined in Malliavin calculus.

We say that the noise is white if $\mu(d\xi)=(2\pi)^{-d}d\xi$. In this case, $\cP_0=L^2(\bR^d)$ and $\langle \varphi,\psi\rangle_{\cP_0}=\langle \varphi,\psi\rangle_{L^2(\bR^d)}$.

We are interested in two cases:
\begin{description}
\item[(I)] ({\em the regular case}) $\gamma=\cF \mu$ is a non-negative locally-integrable function and
    $\mu$ satisfies Dalang's condition:
\begin{equation}
\label{D-cond}
\int_{\bR^d}\frac{1}{1+|\xi|^2}\mu(d\xi)<\infty,
\tag{D}
\end{equation}
    or the noise is white and $d=1$;
\item[(II)] ({\em the rough case}) $d=1$ and $\mu(d\xi)=c_H |\xi|^{1-2H}d\xi$ for some $H \in (\frac{1}{4},\frac{1}{2})$, where
    \begin{equation}
    \label{def-cH}
    c_{H}=\frac{\Gamma(2H+1)\sin(\pi H)}{2\pi}.
    \end{equation}
\end{description}

Alternatively, the inner product in $\cP_0$ can be represented as follows: in Case {\bf (I)},
\[
\langle \varphi, \psi\rangle_{\cP_0}=\int_{\bR^d} \int_{\bR^d}\gamma(x-y) \varphi(x) \psi(y)dxdy,
\]
with the convention $\gamma=\delta_0$ if the noise is white; in Case {\bf (II)},
\begin{equation}
\label{Gagliardo}
\langle \varphi, \psi\rangle_{\cP_0}=C_H \int_{\bR}\int_{\bR} \big(\varphi(x)-\varphi(y)\big)\big(\psi(x)-\psi(y)\big)|x-y|^{2H-2}dxdy,
\end{equation}
where $C_H=\frac{H(1-2H)}{2}$. We say that \eqref{Gagliardo} is the {\em Gagliardo representation}.

\medskip

A typical example of a function $\gamma$ as in Case {\bf (I)} is the Riesz kernel:
\[
\gamma(x)=|x|^{-\beta} \quad \mbox{for some} \quad \beta \in (0,d).
\]
 In this case, $\mu(d\xi)=C_{d,\beta}|\xi|^{-(d-\beta)}d\xi$, and (D) holds if and only if $\beta<2$.

\medskip

For $x \in \bR^d$ and $t>0$, let $p_t(x) = (2 \pi t)^{-d/2} \exp(-|x|^2/(2t))$ be the heat kernel, where $|\cdot|$ is the Euclidean norm. Note that the Fourier transform of $p_t$ is
\begin{align*}
	\cF p_t(\xi) = \exp \left( -\frac{t|\xi|^2}{2} \right), \quad \mbox{for all} \quad \xi \in \bR^d.
\end{align*}

A process $u=u(t,x);t\geq 0,x \in \bR^d\}$ is called a (Skorohod) {\bf solution} of \eqref{eq-heat} if it satisfies the following integral equation:
\begin{align}
\label{int-pam}
	u(t,x) = 1 + \int_0^t \int_{\bR^d} p_{t-s}(x-y) u(s,y) W(\delta y) ds,
\end{align}
where $W(\delta y)$ denotes the Skorokod integral with respect to $W$. The definition of this integral is given in Section \ref{section-Malliavin} below.

We are interested in the asymptotic behavior as $R\to \infty$ of the spatial integral:
\[
F_R(t)=\int_{B_R}\big(u(t,x)-1\big)dx,
\]
where $B_R=\{x \in \bR^d;|x|<R\}$. We let $\omega_d$ be the Lebesque measure of $B_1$. We denote
\[
\sigma_R^2(t)={\rm Var}\big(F_R(t)\big) \quad \mbox{and} \quad \rho_{t,s}(x-y)=\bE[(u(t,x)-1)(u(s,y)-1)].
\]


We consider the following assumptions:

\medskip

\noindent {\bf Assumption A.} ({\em integrable kernel}) $\gamma=\cF \mu$ is a non-negative function, $\gamma \in L^1(\bR^d)$ and $\mu$ satisfies condition (D), or $d=1$ and the noise is white.

\medskip

\noindent {\bf Assumption B.} ({\em Riesz kernel}) $\gamma(x)=|x|^{-\beta}$ for some $\beta \in (0,d \wedge 2)$.

\medskip

\noindent {\bf Assumption C.} ({\em rough noise}) $d=1$ and $\mu(d\xi)=c_H|\xi|^{1-2H}d\xi$ for some $H \in (\frac{1}{4},\frac{1}{2})$, where the constant $c_H$ is given by \eqref{def-cH}.

\medskip

We write $f(R) \sim g(R)$ if $f(R)/g(R) \to 1$ as $R \to \infty$.

\medskip

The main results of this article are summarized in the following theorem.

\begin{theorem}
\label{main}
(i) {\rm (Limiting Covariance)} For any $t>0$ and $s>0$,
\[
\bE[F_R(t)F_R(s)] \sim
\begin{cases}
R^{d}K(t,s)  & \mbox{under Assumptions A or C} \\
R^{2d-\beta}K(t,s) & \mbox{under Assumption B}
\end{cases}
\]
where
\[
K(t,s)=
\begin{cases}
\omega_d \int_{\bR^d}\rho_{t,s}(z)dz    & \mbox{under Assumption A} \\
ts \int_{B_1^2} |x-x'|^{-\beta} dxdx' & \mbox{under Assumption B} \\
2\int_{\bR}\bE[e^{\cI_{t,s}^{1,2}(z)}-\cI_{t,s}^{1,2}(z)-1]dz & \mbox{under Assumption C}
\end{cases}
\]
where $\cI_{t,s}^{1,2}(z)$ is defined by \eqref{def-cI} below; in particular,
$\sigma_R^2(t) \sim R^d K(t,t)$ under Assumptions A or C, and
$\sigma_R^2(t) \sim R^{2d-\beta} K(t,t)$ under Assumption B.

(ii) {\rm (Quantitative Central Limit Theorem)} For any $t>0$,
\[
d_{\rm TV}\left( \frac{F_R(t)}{\sigma_R(t)},Z\right) \leq
\begin{cases}
C_t R^{-d/2}  & \mbox{under Assumptions A or C}, \\
C_t R^{-\beta/2} & \mbox{under Assumption B}
\end{cases}
\]
where $Z\sim N(0,1)$ and $C_{t}>0$ is a constant depending on $t$.

(iii) {\rm (Functional Central Limit Theorem)} The process $\{Q_R(t)\}_{t\geq 0}$ given by:
\[
Q_R(t)=
\begin{cases}
R^{-d/2} F_R(t) & \mbox{under Assumptions A or C}, \\
R^{-d+\beta/2} F_R(t) & \mbox{under Assumption B}
\end{cases}
\]
has a continuous modification which converges in distribution in $C[0,\infty)$ as $R \to \infty$, to a zero-mean Gaussian process $\{\cG(t)\}_{t\geq 0}$ with covariance
$ \bE[\cG(t)\cG(s)]=K(t,s)$.
\end{theorem}

This article is organized as follows. In Section 2, we give some preliminaries about Malliavin calculus. The proof of Theorem \ref{main} is given Section 3 in the regular case (under Assumptions A or B), respectively Section 4 in the rough case (under Assumption C). In each case, we prove the existence of (Skorohod) solution, and we give some estimates on its Malliavin derivatives of first and second order, which play a key role in our developments. The appendices contain some auxiliary results needed in the sequel.

\section{Malliavin Calculus Preliminaries}
\label{section-Malliavin}

In this section, we include some preliminaries on Malliavin calculus and we establish the existence and uniqueness of the solution to \eqref{eq-heat}.

Since $W=\{W(\varphi);\varphi \in \cP_0\}$ is an isonormal Gaussian process,
every square-integrable random variable $F$
which is measurable with respect to $W$ has the Wiener chaos expansion:
\begin{equation}
\label{F-chaos}
F=E(F)+\sum_{n \geq 1}I_n(f_n) \quad \mbox{for some} \quad f_n \in \cP_0^{\otimes n},
\end{equation}
where $\cP_0^{\otimes n}$ is the
$n$-th tensor product of $\cP_0$ and $I_n$
is the multiple Wiener integral with respect to $W$.
By the orthogonality of the Wiener chaos spaces,
\[
 E[I_n(f)I_m(g)]=\left\{
\begin{array}{ll} n! \, \langle \widetilde{f}, \widetilde{g} \rangle_{\cP_{0}^{\otimes n}} & \mbox{if $n=m$} \\
0 & \mbox{if $n \not=m$}
\end{array} \right.
\]
where $\widetilde{f}$ is the symmetrization of $f$ in all $n$
variables:
$$\widetilde{f}(x_1,\ldots,x_n)=\frac{1}{n!}\sum_{\rho \in S_n}f(x_{\rho(1)},\ldots,x_{\rho(n)}),$$
and $S_n$ is the set of all permutations of $\{1,
\ldots,n\}$. It can be proved that:
\begin{equation}
\label{rough-bound}
\|\widetilde f\|_{\cP_0^{\otimes n}} \leq \|f\|_{\cP_0^{\otimes n}},
\end{equation}

If $F$ has the chaos expansion \eqref{F-chaos}, then
$$E|F|^2=\sum_{n \geq 0}E|I_n(f_n)|^2=\sum_{n \geq 0}n! \, \|\widetilde{f}_n\|_{\cH^{\otimes n}}^{2}.$$

Let $\cS$ be the class of ``smooth'' random variables, i.e variables of the form
\begin{equation}
\label{form-F}F=f(W(\varphi_1),\ldots, W(\varphi_n)),
\end{equation} where $f \in C_{b}^{\infty}(\bR^n)$, $\varphi_i \in \cP_0$, $n \geq 1$, and
$C_b^{\infty}(\bR^n)$ is the class of bounded $C^{\infty}$-functions
on $\bR^n$, whose partial derivatives of all orders are bounded. The
{\em Malliavin derivative} of $F$ of the form (\ref{form-F}) is the
$\cP_0$-valued random variable given by:
$$DF:=\sum_{i=1}^{n}\frac{\partial f}{\partial x_i}(W(\varphi_1),\ldots,
W(\varphi_n))\varphi_i.$$ We endow $\cS$ with the norm
$\|F\|_{\bD^{1,2}}:=(E|F|^2)^{1/2}+(E\|D F \|_{\cP_0}^{2})^{1/2}$. The
operator $D$ can be extended to the space $\bD^{1,2}$, the
completion of $\cS$ with respect to $\|\cdot \|_{\bD^{1,2}}$.

The {\em divergence operator} $\delta$ is the adjoint of
the operator $D$. The domain of $\delta$, denoted by $\mbox{Dom} \
\delta$, is the set of $u \in L^2(\Omega;\cP_0)$ such that
$$|E \langle DF,u \rangle_{\cH}| \leq c (E|F|^2)^{1/2}, \quad \forall F \in \bD^{1,2},$$
where $c$ is a constant depending on $u$. If $u \in {\rm Dom} \
\delta$, then $\delta(u)$ is the element of $L^2(\Omega)$
characterized by the following duality relation:
\begin{equation}
\label{duality}
E(F \delta(u))=E\langle DF,u \rangle_{\cP_0}, \quad
\forall F \in \bD^{1,2}.
 \end{equation}
In particular, $E(\delta(u))=0$. If $u \in \mbox{Dom} \ \delta$, we
use the notation
$$\delta(u)=\int_{\bR^d}u(x) W(\delta x),$$
and we say that $\delta(u)$ is the {\em Skorohod integral} of $u$
with respect to $W$.

If $F$ has the chaos expansion \eqref{F-chaos}, we define the {\em Ornstein-Uhlenbeck generator}
\[
LF=\sum_{n\geq 1}n I_n(f_n)
\]
provided that the series converges in $L^2(\Omega)$.
It can be proved that
$F \in {\rm Dom}\ L$ if and only if $F \in \bD^{1,2}$ and $DF \in {\rm Dom} \ \delta$; in this case, $LF=-\delta (D F)$.
The pseudo-inverse $L^{-1}$ of $L$ is defined by
\[
L^{-1}F=\sum_{n\geq 1}\frac{1}{n} I_n(f_n).
\]
For any $F \in \bD^{1,2}$ with $\bE(F)=0$, the process $u=-D L^{-1}F$ belongs to ${\rm Dom} \ \delta$
and
\begin{equation}
\label{F-deltaD}
F=\delta(-D L^{-1} F).
\end{equation}
(see e.g. Proposition 6.5.1 of \cite{NN}).

\medskip

We return now to our problem.

If the solution exists, it should be given by the series
\begin{align} \label{eq-chaos expension-heat}
	u(t,x) = 1 + \sum_{n\geq 1} I_n(f_n(\cdot,x;t)),
\end{align}
where the kernel $f_n(\cdot,x;t)$ is given by
\begin{align*}
	f_n(x_1,\ldots,x_n,x;t)
	= \int_{T_n(t)} p_{t-t_n}(x-x_n) \ldots p_{t_2-t_1}(x_2-x_1) dt_1\ldots dt_n,
\end{align*}
and $T_n(t) = \{ (t_1,\ldots,t_n): 0<t_1<\ldots<t_n<t \}$. The necessary and sufficient condition for the existence of the solution is that this series converges in $L^2(\Omega)$, i.e.
\begin{equation}
\label{series-conv}
\sum_{n \geq 1}n! \|\widetilde{f}_n(\cdot,x;t)\|_{\cP_0^{\otimes n}}^2<\infty.
\end{equation}
Under this condition, the solution is unique, and
\[
\bE|u(t,x)|^2=1+\sum_{n\geq 1}n! \|\widetilde{f}_n(\cdot,x;t)\|_{\cP_0^{\otimes n}}^2.
\]

\section{The Regular Case}
\label{section-regular}
In this section we consider Case (I) (the regular case). One of the most important properties which is specific to this is case is a maximal principle:
if $\gamma=\cF \mu$ is a non-negative and non-negative-definite function on $\bR^d$, and $\mu$ satisfies Dalang's condition \eqref{D-cond}, then
\begin{align}
\label{max-principle}
	\sup_{\eta \in \bR^d} \int_{\bR^d} |\cF h(\xi+\eta)|^2 \mu(d\xi)
	= \int_{\bR^d} |\cF h(\xi)|^2 \mu(d\xi),
\end{align}
for any non-negative function $h$ such that $h \in \cP_0 \cap L^1(\bR^d)$ or $\cF h \geq 0$ (see Lemma 2.6 of \cite{BNQZ} and Lemma 3.6 of \cite{CDST}). In particular, this holds for $h(t)=e^{-t|\xi|^2}$ for any $t>0$.

\subsection{Existence of Solution}

In this section, we examine the question of existence of solution.

For any $t>0$, we denote $h_0(t)=1$ and
\begin{align}
\label{def-hn}
h_n(t)=
\int_{T_n(t)} \int_{(\bR^d)^n} \prod_{j=1}^{n} e^{-(t_{j+1}-t_j)|\xi_j|^2} \mu(d\xi_1) \ldots \mu(d\xi_n)d\pmb{t} \quad \mbox{for all} \quad n\geq 1,
\end{align}
where $\pmb{t}=(t_1,\ldots,t_n)$ and we let $t_{n+1}=t$.
We recall the following result.

\begin{lemma}[Lemma 3.8 of \cite{BC}]
\label{Lem-3.8}
Under condition (D), for any $t>0$ and $\gamma >0$,
\begin{align*}
H(t;\gamma):=\sum_{n\ge0} \gamma^n h_n(t) <\infty \quad \mbox{and} \quad
\widetilde{H}(t;\gamma):=\sum_{n\ge0} \sqrt{\gamma^n h_n(t)} <\infty.
\end{align*}
\end{lemma}

By Lemma 2.6 of \cite{chen-kim}, $h_n$ is non-decreasing for any $n\geq 1$. 

%

\begin{theorem} \label{Thm-existence-heat}
Assume that $\gamma$ is non-negative and non-negative definite and the spectral measure $\mu$ satisfies (D), or the noise is white. Then for any $t>0$ and $x \in \bR^d$,
\begin{equation}
\label{norm-fn}
\|f_n(\cdot,x;t)\|_{\cP_0^{\otimes n}}^2 \leq \frac{t^n}{n!}h_n(t).
\end{equation}
Consequently, equation \eqref{eq-heat} has a unique solution $u$, and for any $p \geq 2$ and $T>0$,
\begin{equation}
\label{p-mom-u}
\sup_{(t,x) \in [0,T] \times \bR^d}\|u(t,x)\|_p<\infty.
\end{equation}
\end{theorem}

\begin{proof}
The Fourier transform of $f_n(\cdot,x;t)$ is given by:
\[
\cF f_n(\cdot,x;t)(\xi_1,\ldots,\xi_n)=e^{-i(\xi_1+\ldots+\xi_n)\cdot x} \int_{T_n(t)} \prod_{j=1}^{n}e^{-\frac{1}{2}(t_{j+1}-t_j)|\xi_1+\ldots+\xi_j|^2}d\pmb{t},
\]
where $\pmb{t}=(t_1,\ldots,t_n)$ and $t_{n+1}=t$.
By Cauchy-Schwarz inequality,
\begin{equation}
\label{Jn-CS}
\|f_n(\cdot,x;t)\|_{\cP_0^{\otimes n}}^2 =\int_{(\bR^d)^n}|\cF f_n(\cdot,x;t)(\xi_1,\ldots,\xi_n)|^2 \mu(d\xi_1)\ldots \mu(d\xi_n) \leq \frac{t^n}{n!}J_n(t),
\end{equation}
where
\begin{equation}
\label{def-Jn}
J_n(t)=\int_{T_n(t)}\int_{(\bR^d)^n} \prod_{j=1}^{n}e^{-(t_{j+1}-t_j)|\xi_1+\ldots+\xi_j|^2}\mu(d\xi_1)\ldots \mu(d\xi_n) d\pmb{t}.
\end{equation}
By the maximal principle \eqref{max-principle},
\begin{equation}
\label{Jn-hn}
J_n(t) \leq h_n(t).
\end{equation}
Relation  \eqref{norm-fn} follows.
By the rough bound \eqref{rough-bound} and Lemma \ref{Lem-3.8}, we have:
\[
\sum_{n\geq 1}n! \|\widetilde{f}_n(\cdot,x;t)\|_{\cP_0^{\otimes n}}^2 \leq \sum_{n\geq 1}n! \|f_n(\cdot,x;t)\|_{\cP_0^{\otimes n}}^2 \leq \sum_{n\geq 0}t^n J_n(t) \leq \sum_{n\geq 0}t^n h_n(t)=H(t,t)<\infty.
\]
This proves the existence and uniqueness of the solution. By Minkowski's inequality and hypercontractivity,
\[
\|u(t,x)\|_p \leq \sum_{n\geq 0}(p-1)^{n/2}(n!)^{1/2}\|\widetilde{f}_n(\cdot,x;t)\|_{\cP_0^{\otimes n}} \leq \sum_{n\geq 0}(p-1)^{n/2} t^{n/2}\sqrt{h_n(t)}=\widetilde{H}(t;(p-1)t).
\]
The last term is bounded by $\widetilde{H}(T;(p-1)T)$, for any $t \in [0,T]$ and $x \in \bR^d$.
\end{proof}

\begin{remark}[Comparison with the white noise in time]
{\rm
Consider the parabolic Anderson model with Gaussian noise $\fX$ which is white noise in time and has the same spatial covariance structure as $W$:
\begin{align}
\label{pam1}
	\begin{cases}
		\dfrac{\partial v}{\partial t} (t,x)
		= \Delta v(t,x) + v(t,x) \dot{\fX}(t,x), \
		t>0, \ x \in \bR^d \\
		v(0,x) = 1,
	\end{cases}
\end{align}
More precisely, $\fX=\{\fX(\varphi);\varphi \in \cD(\bR_{+} \times \bR^d)\}$ is a zero-mean Gaussian process with covariance
\[
\bE[\fX(\varphi)\fX(\psi)]=\int_{\bR_{+}}
\int_{(\bR^d)^2}\gamma(x-y)\varphi(t,x)\psi(t,y)dxdydt=:\langle \varphi,\psi \rangle_{\cH_0},
\]
We let $\cH_0$ be the completion of $\cD(\bR_{+} \times \bR^d)$ with respect to the inner product $\langle \varphi,\psi \rangle_0$. Then $\cH_0$ is isomorphic to $L^2(\bR_{+};\cP_0)$.
If (D) holds, equation \eqref{pam1} has a unique solution which has the chaos expansion:
\[
v(t,x)=1+\sum_{n\geq 1}I_n^{\fX}(f_n(\cdot,t,x))
\]
where $I_n^{\fX}$ is the multiple integral with respect to $\fX$ and the kernel $f_n(\cdot,t,x)$ is given by
\[
f_n(t_1,x_1,\ldots,t_n,x_n,t,x)=p_{t-t_n}(x-x_n)\ldots p_{t_2-t_1}(x_2-x_1)1_{\{0<t_1<\ldots<t_n<t\}}.
\]
It can be proved that $\|f_n(\cdot,t,x)\|_{\cH_0^{\otimes n}}^2=J_n(t)$.
Using \eqref{Jn-hn}, followed by Lemma 3.3 of \cite{HHNT}, we obtain the following inequality:
\begin{equation}
\label{norm-fn-white}
\|f_n(\cdot,t,x)\|_{\cH_0^{\otimes n}}^2 \leq h_n(t) \leq \sum_{\ell=0}^{n} \binom{n}{\ell} \frac{C_N^{n-\ell}}{\ell!}(tD_N)^{\ell},
\end{equation}
for any $N>0$, where
\[
C_N=\int_{|\xi|> N}\frac{1}{|\xi|^2}\mu(d\xi) \quad \mbox{and} \quad
D_N=\int_{\{|\xi|\leq N\}}\mu(d\xi).
\]
Note that $\bE|v(t,x)|^2=\sum_{n\geq 0} J_n(t)$, whereas $\bE|u(t,x)|^2 \leq \sum_{n\geq 0}t^n J_n(t)$.}
\end{remark}

\subsection{Estimates on the Malliavin derivatives}

In this subsection, we give the chaos expansions for the first and second Malliavin derivatives of the solution to \eqref{eq-heat}. For each of these derivatives, we establish a moment estimate which shows that the first term in the chaos expansion dominates the rest.

We begin with the first Malliavin derivative. We fix $t>0$ and $x \in \bR^d$. We will prove that for any $z \in \bR^d$, we have the following chaos expansion:
\begin{equation}
\label{D-series}
	D_{z}u(t,x)
	=\sum_{n\geq 1} n I_{n-1}(\widetilde{f}_n(\cdot,z,x;t))=:\sum_{n\geq 1}A_n(z,x;t) \quad \mbox{in $L^2(\Omega)$}.
\end{equation}

In order to do this, we first note that, as in the case of the wave equation (studied in \cite{BY}), we have the decomposition:
\begin{equation} \label{decomp'}
	\widetilde{f}_n(\cdot,z,x;t)
	=\frac{1}{n}\sum_{j=1}^n h_j^{(n)}(\cdot,z,x;t),
\end{equation}
where $h_j^{(n)}(\cdot,z,x;t)$ is the symmetrization of the function $f_j^{(n)}(\cdot,z,x;t)$ given by:
\begin{align*}
	& f_j^{(n)}(x_1,\ldots,x_{n-1},z,x;t) = f_n(x_1,\ldots,x_{j-1},z,x_{j},\ldots,x_{n-1},x;t)  \\
	& \quad =\int_{\{0<t_1<\ldots<t_{j-1}<r<t_j<\ldots<t_{n-1}<t\}}
	p_{t-t_{n-1}}(x-x_{n-1}) \ldots p_{t_j-r}(x_j-z) p_{r-t_{j-1}}(z-x_{j-1}) \ldots \\
	& \qquad \qquad \qquad \qquad \qquad p_{t_2-t_1}(x_2-x_1) dt_1 \ldots dt_{n-1}dr
\end{align*}

We will use the following functions:
\begin{align}
\label{def-fk}
f_n(t_1,x_1,\ldots,t_n,x_n,t,x)&=p_{t-t_n}(x-x_n)\ldots p_{t_2-t_1}(x_2-x_1) \\
\label{def-gk}
g_n(t_1,x_1,\ldots,t_n,x_n,r,z,t,x)&=p_{t-t_n}(x-x_n)\ldots p_{t_2-t_1}(x_2-x_1)p_{t_1-r}(x_1-z),
\end{align}
We use the following notational convention: for any function $f: (\bR_{+} \times \bR^d)^n \to \bR$, we denote
\[
f(\pmb{t_n},\pmb{x_n}):=f(t_1,x_1,\ldots,t_n,x_n).
\]
where $\pmb{t_n}=(t_1,\ldots,t_n)$ and $\pmb{x_n}=(x_1,\ldots,x_n)$.
We denote
$\pmb{t_{j:n}}=(t_{j},\ldots,t_{n})$ for any $j\leq n$.


We let $g_0(r,z,t,x)=p_{t-r}(x-z)$. The following decomposition holds:
\begin{align}
\label{decomp-fjn}
f_j^{(n)}(\pmb{x_{n-1}},z,x;t)&=\int_0^t
\left(\int_{T_{j-1}(r)} f_{j-1}
	(\pmb{t_{j-1}},\pmb{x_{j-1}},r,z) d\pmb{t_{j-1}} \right)\\
\nonumber
& \quad \left(\int_{\{r<t_j<\ldots<t_{n-1}<t\}}
g_{n-j}(\pmb{t_{j:n-1}},\pmb{x_{j:n-1}},r,z,t,x)d\pmb{t_{j:n-1}}\right)
dr.
\end{align}

Recall that $\cH_0=L^2(\bR_{+};\cP_0)$ is the Hilbert space associated with the white noise in time. By Cauchy-Schwarz inequality and \eqref{norm-fn-white},
\begin{align}
\nonumber
	\left(\int_{T_n(t)} \|f_{n}
	(\pmb{t_{n}},\bullet,t,x)\|_{\cP_0^{\otimes n}}d\pmb{t_n}\right)^2  & \leq \dfrac{t^n}{n!} \|f_{n}(\cdot,t,x)\|_{\cH_0^{\otimes n}}^2 \leq \dfrac{t^n}{n!}\sum_{\ell=0}^{n}\binom{n}{\ell}\frac{C_N^{n-\ell}}{\ell!}
(tD_N)^{\ell}\\
\label{in1}
& \leq \dfrac{(2t)^{n}}{n!}\sum_{\ell=0}^{n}\frac{C_N^{n-\ell}}{\ell!}
(tD_N)^{\ell}.
\end{align}
A similar inequality holds for $g_n$.
To see this, we recall the following estimate which was proved in \cite{NXZ}: 
\begin{equation}
\label{3-17-NXZ}
\|g_n(\cdot,r,z,t,x)\|_{\cH_0}^{2} \leq p_{t-r}^2(x-z) \sum_{\ell=0}^{n} \binom{n}{\ell} \frac{(4C_N)^{n-\ell}}{\ell!}[(t-r)D_N]^{\ell}
\end{equation}
(see the proof of relation (3.17) of \cite{NXZ}). By convention, the sum is equal to $1$ if $n=0$. Applying Cauchy-Schwarz inequality, we obtain:
\begin{align}
\nonumber
	& \left(\int_{\{r<t_1<\ldots<t_{n}<t\}} \|g_{n}
	(\pmb{t_{n}},\bullet,r,z,t,x)\|_{\cP_0^{\otimes n}} d\pmb{t_{n}}\right)^2  \leq \dfrac{(t-r)^{n}}{n!}
	\|g_{n}(\cdot,r,z,t,x)\|_{\cH_0^{\otimes n}}^2 \\
\nonumber
& \quad \quad \quad \leq \dfrac{(t-r)^{n}}{n!}p_{t-r}^2(x-z)\sum_{\ell=0}^{n}\binom{n}{\ell}
\frac{(4C_N)^{n-\ell}}{\ell!}
[(t-r)D_N]^{\ell} \\
\label{in2}
& \quad \quad \quad \leq \dfrac{(2t)^{n}}{n!}p_{t-r}^2(x-z)\sum_{\ell=0}^{n}
\frac{(4C_N)^{n-\ell}}{\ell!}
(tD_N)^{\ell}.
\end{align}

Here is the first result of this section.

\begin{theorem}
\label{main-th3}
Suppose that $\gamma=\cF \mu$ is a non-negative locally-integrable function and $\mu$ satisfies (D), or $d=1$ and the noise is white. Then the series in \eqref{D-series} converges in $L^2(\Omega)$, and for any $p\geq 2$, $T>0$, $t \in [0,T]$ and $x,z \in \bR^d$,
\begin{equation}
\label{D-bound}
	\|D_{z}u(t,x)\|_p \leq C_T \int_0^t p_{t-r}(x-z)dr,
\end{equation}
where the constant $C_T$ depends on $(T,p,\gamma)$.
\end{theorem}

\begin{proof}
We first prove the convergence of series \eqref{D-series} in $L^2(\Omega)$. By relation (26) of \cite{BY} (which obviously continues to hold for the heat equation), we have:
\begin{equation}
\label{bound-An}
\bE|A_n(z,x;t)|^2 \leq n! \sum_{j=1}^{n} \|f_j^{(n)}(\cdot,z,x;t)\|_{\cP_0^{\otimes (n-1)}}^2.
\end{equation}

We estimate separately the terms of this sum.
 By \eqref{decomp-fjn} and Minkowski's inequality,
\begin{align}
\label{eq-norm-f_j^n-heat}
	& \|f_j^{(n)}(\cdot,z,x;t)\|_{\cP_0^{\otimes (n-1)}}
	\leq \int_0^t
	\left(\int_{T_{j-1}(r)} \|f_{j-1}
	(\pmb{t_{j-1}},\bullet,r,z)\|_{\cP_0^{\otimes (j-1)}} d\pmb{t_{j-1}}\right) \nonumber \\
	& \qquad \qquad \qquad \qquad \qquad \left(\int_{\{r<t_j<\ldots<t_{n-1}<t\}}	\|g_{n-j}(\pmb{t_{j:n-1}},\bullet,r,z,t,x)\|_{\cP_0^{\otimes (n-j)}} d\pmb{t_{j:n-1}} \right)dr.
\end{align}
For each $r \in (0,t)$ fixed, we estimate separately each of the inner integrals above. (By convention, the first integral is equal 1 if $j=1$, and the second one is equal to $p_{t-r}(x-z)$ if $j=n$.)
We use \eqref{in1} for the first integral and \eqref{in2} for the second integral. We obtain:
\begin{align}
\nonumber
& \|f_j^{(n)}(\cdot,z,x;t)\|_{\cP_0^{\otimes (n-1)}}^2 \leq \left(\int_0^t p_{t-r}(x-z)dr \right)^2 \\
\label{fjn-step1}
& \quad
\frac{(2t)^{n-1}}{(j-1)!(n-j)!}
\left(\sum_{\ell=0}^{j-1}\frac{C_N^{j-1-\ell}}{\ell!}
(tD_N)^{\ell}\right) \left(\sum_{\ell=0}^{n-j}
\frac{(4C_N)^{n-j-\ell}}{\ell!}
(tD_N)^{\ell} \right).
\end{align}
We consider two cases.

{\em Case 1.} Suppose that there exists $N_0 \geq 1$ such that $C_{N_0}=0$. Then $C_N=0$ for all $N \geq N_0$. Let $N \geq N_0$ be arbitrary. Then the only non-zero terms in the two sums on the right-hand side of \eqref{fjn-step1} are those corresponding to $\ell=j-1$ (for the first sum), respectively $\ell=n-j$ (for the second sum). Therefore, relation \eqref{fjn-step1} becomes:
\begin{align*}
 \|f_j^{(n)}(\cdot,z,x;t)\|_{\cP_0^{\otimes (n-1)}}^2 \leq
 \left(\int_0^t p_{t-r}(x-z)dr \right)^2 \frac{(2t^2 D_N)^{n-1}}{[(j-1)!]^2 [(n-j)!]^2}
\end{align*}
Taking the sum for all $j=1,\ldots,n$ we obtain:
\begin{align*}
\sum_{j=1}^n \|f_j^{(n)}(\cdot,z,x;t)\|_{\cP_0^{\otimes (n-1)}}^2
& \leq \left(\int_0^t p_{t-r}(x-z)dr \right)^2 \frac{(2t^2 D_N)^{n-1}}{[(n-1)!]^2} \sum_{j=1}^{n-1}\binom{n-1}{j-1}^2\\
& \leq \left(\int_0^t p_{t-r}(x-z)dr \right)^2 \frac{(8t^2 D_N)^{n-1}}{[(n-1)!]^2},
\end{align*}
using the fact that $\sum_{j=1}^{n}\binom{n-1}{j-1}^2 \leq \left(\sum_{j=1}^{n}\binom{n-1}{j-1}\right)^2=4^{n-1}$.
Coming back to \eqref{bound-An}, we get:
\[
\bE|A_n(z,x;t)|^2 \leq n\left(\int_0^t p_{t-r}(x-z)dr \right)^2  \frac{(8t^2 D_N)^{n-1}}{(n-1)!}.
\]
This proves that $\sum_{n\geq 1}A_n(z,x;t)$ converges in $L^2(\Omega)$, and
\[
\bE|D_{z}u(t,x)|^2=\sum_{n\geq 1}\bE|A_n(z,x;t)|^2 \leq \left(\int_0^t p_{t-r}(x-z)dr \right)^2 \exp(16t^2 D_N).
\]
Relation \eqref{D-bound} follows by hypercontractivity:
\[
\|D_{z}u(t,x)\|_p \leq \sum_{n\geq 1}(p-1)^{(n-1)/2} \|A_n(z,x;t)\|_2 \leq C_{t,p} \left(\int_0^t p_{t-r}(x-z)dr \right),
\]
where $C_{t,p}=\exp(c (p-1)t^2 D_N)$ for some constant $c>0$.

\medskip
{\em Case 2.} Suppose that $C_{N}>0$ for all $N>0$. Bounding the two sums on the right-hand side of \eqref{fjn-step1} by $C_{N}^{j-1}e^{C_N^{-1}tD_N}$, respectively $(4C_N)^{n-j}e^{(4C_N)^{-1}tD_N}$, we obtain:
\[
\|f_j^{(n)}(\cdot,z,x;t)\|_{\cP_0^{\otimes (n-1)}}^2 \leq \left(\int_0^t p_{t-r}(x-z)dr \right)^2 \frac{(2t)^{n-1}}{(j-1)!(n-j)!}(4C_N)^{n-1}e^{2t C_N^{-1}D_N}.
\]
Taking the sum over $j=1,\ldots,n$ and using the fact that $\sum_{j=1}^{n}\binom{n-1}{j-1}=2^{n-1}$, we get:
\[
\sum_{j=1}^n\|f_j^{(n)}(\cdot,z,x;t)\|_{\cP_0^{\otimes (n-1)}}^2 \leq \left(\int_0^t p_{t-r}(x-z)dr \right)^2 \frac{(16t C_N)^{n-1}}{(n-1)!}e^{2t C_{N}^{-1}D_N}.
\]
Coming back to \eqref{bound-An} and using the fact that $n \leq 2^{n-1}$, we get:
\[
\bE|A_n(z,x;t)|^2 \leq \left(\int_0^t p_{t-r}(x-z)dr \right)^2 (32t C_N)^{n-1}
e^{2t C_{N}^{-1}D_N}.
\]
Due to condition (D), $\lim_{N \to \infty}C_N=0$. Hence, we can choose $N=N_T$ such that $$32TC_N<1.$$ Then for any $t \in [0,T]$, the series $\sum_{n\geq 1}A_n(z,x;t)$ converges in $L^2(\Omega)$, and
\[
\bE|D_zu(t,x)|^2=\sum_{n\geq 1}\bE|A_n(z,x;t)|^2 \leq \left(\int_0^t p_{t-r}(x-z)dr \right)^2 \frac{1}{1-32 TC_N}e^{2T C_{N}^{-1}D_N}.
\]
Relation \eqref{D-bound} follows by hypercontractivity, as in Case 1 above.

\end{proof}

\medskip
We now examine the second Malliavin derivative. We fix $t>0$ and $x \in \bR^d$.
We will show below that for any $w,z \in \bR^d$, we have the following chaos expansion:
\begin{equation}
\label{D2-series}
	D_{w,z}^2 u(t,x)
	=\sum_{n\geq 2} n(n-1) I_{n-2}(\widetilde{f}_n(\cdot,w,z,x;t))=:\sum_{n\geq 2}B_n(w,z,x;t) \quad \mbox{in} \quad L^2(\Omega).
\end{equation}	

As in the case of the wave equation (see \cite{BY}), we note that:
\[
 \widetilde{f}_n(\cdot,w,z,x;t)=\frac{1}{n(n-1)} \sum_{i,j=1,i\not=j}^{n}
 h_{ij}^{(n)}(\cdot,w,z,x;t),
\]
 where $h_{ij}^{(n)}(\cdot,w,z,x;t)$ is the symmetrization of the function
 $f_{ij}^{(n)}(\cdot,w,z,x;t)$ defined as follows. If $i<j$,
\begin{align*}
	& f_{ij}^{(n)}(x_1,\ldots,x_{n-2},w,z,x;t) = f_n(x_1, \ldots, x_{i-1}, w, x_i,	\ldots, x_{j-2}, z, x_{j-1}, \ldots, x_{n-2},x;t) \\
	=&	\int_{\{t_1< \ldots< t_{i-1}< \theta< t_i< \ldots< t_{j-2}< r< t_{j-1}< \ldots< t_{n-2}<t\}}
	p_{t-t_{n-2}}(x-x_{n-2}) \ldots p_{t_{j-1}-r}(x_{j-1}-z) \\
	& p_{r-t_{j-2}}(z-x_{j-2}) \ldots p_{t_i-\theta}(x_i-w)p_{\theta-t_{i-1}}(w-x_{i-1}) \ldots p_{t_2-t_1}(x_2-x_1) dt_1 \ldots dt_{n-2}drd\theta,
\end{align*}
and we have the decomposition:
\begin{align}
\label{dec-fij}
f_{ij}^{(n)}(\pmb{x_{n-2}},w,z,x;t)& =\int_{\{0<\theta<r<t\}}
\left(\int_{T_{i-1}(\theta)} f_{i-1} (\pmb{t_{i-1}},\pmb{x_{i-1}},\theta,w) d\pmb{t_{i-1}}\right) \\
\nonumber
& \quad \left(\int_{\{ \theta <t_i<\ldots< t_{j-2}<r \}}
g_{j-i-1}(\pmb{t_{i:j-2}},\pmb{x_{i:j-2}},\theta,w,r,z) d\pmb{t_{i:j-2}}\right) \\
\nonumber
& \quad \left(\int_{\{ r<t_{j-1}<\ldots<t_{n-2}<t \}}
g_{n-j}(\pmb{t_{j-1:n-2}},\pmb{x_{j-1:n-2}},r,z,t,x) \pmb{t_{j-i:n-2}}\right)d\theta dr.
\end{align}

If $j<i$,
\[
 f_{ij}^{(n)}(x_1,\ldots,x_{n-2},w,z,x;t)=f_n(x_1,\ldots,x_{j-1},z,x_{j},
\ldots,x_{i-2},w,x_{i-1},\ldots,x_{n-2},x;t)
\]
In both cases, $w$ is on position $i$ and $z$ is on position $j$.

After introducing these notations, we are now ready to prove the following result.

\begin{theorem}
\label{Thm-D2-bound}
Under the hypotheses of Theorem \ref{main-th3}, the series in \eqref{D2-series} converges in $L^2(\Omega)$, and for any $p \geq 2$, $T>0$, $t\in [0,T]$ and $x,w,z \in \bR^d$,
\begin{equation}
\label{bound-D2}
	\|D_{w,z}^2 u(t,x) \|_p \leq C_T \widetilde{f}_2(w,z,x;t),
\end{equation}
where $C_T>0$ is a constant that depends on $(T,p,\gamma)$.
\end{theorem}

\begin{proof}
As in the case of Theorem \ref{main-th3}, the main effort will be dedicated to proving that the series \eqref{D2-series} converges in $L^2(\Omega)$. To this end, we will use the following estimate, which is proved exactly as in the case of the wave equation (see \cite{BY}):
\begin{equation}
\label{bound-Bn}
	\bE |B_n(w,z,x;t)|^2
\leq  n! \sum_{i,j=1,i\not=j}^n \left\| f_{ij}^{(n)}(\cdot,w,z,x;t)) \right\|_{\cP_0^{\otimes (n-2)}}^2,
\end{equation}

We proceed with the estimation of $\| f_{ij}^{(n)}(\cdot,w,z,x;t)) \|_{\cP_0^{\otimes (n-2)}}^2$. Let $1\leq i<j\leq n$ be arbitrary.
By Minkowski's inequality and \eqref{dec-fij},
\begin{align} \label{ineq-fijn}
	& \big\| f_{ij}^{(n)}(\cdot,w,z,x;t) \big\|_{\cP_0^{\otimes (n-2)}}
	\leq  \int_{0<\theta<r<t}
	\left(\int_{T_{i-1}(\theta)} \big\| f_{i-1} (
\pmb{t_{i-1}},\bullet,\theta,w) \big\|_{\cP_0^{\otimes (i-1)}}  d\pmb{t_{i-1}} \right) \nonumber \\
	& \quad \quad \quad \left(\int_{\{\theta<t_i<\ldots< t_{j-2}<r\}} \big\| g_{j-i-1}(\pmb{t_{i:j-2}},\bullet,\theta,w,r,z) \big\|_{\cP_0^{\otimes (j-i-1)}} d\pmb{t_{i:j-2}}\right) \nonumber \\
	& \quad \quad \quad  \left(\int_{\{r<t_{j-1}<\ldots<t_{n-2}<t\}} \big\| g_{n-j}(\pmb{t_{j-1:n-2}},\bullet,r,z,t,x) \big\|_{\cP_0^{\otimes (n-j)}} d\pmb{t_{j-1:n-2}} \right)d\theta dr.
\end{align}
We estimate each of the 3 inner integrals above using \eqref{in1} and \eqref{in2}. We have
\begin{align}
\nonumber
& \big\| f_{ij}^{(n)}(\cdot,w,z,x;t) \big\|_{\cP_0^{\otimes (n-2)}}^2 \leq \\
\nonumber
& \quad \left( \int_{0<\theta<r<t} p_{t-r}(x-z)p_{r-\theta}(z-w)drd\theta\right)^2 \frac{(2t)^{n-2}}{(i-1)!(j-i-1)!(n-j)!} \\
\label{f-ijn-1}
& \quad \left( \sum_{\ell=0}^{i-1} \frac{C_N^{i-1-\ell}}{\ell!} (t D_N)^{\ell}\right) \left( \sum_{\ell=0}^{j-i-1} \frac{(4C_N)^{j-i-1-\ell}}{\ell!}(tD_N)^{\ell}\right) \left( \sum_{\ell=0}^{n-j}\frac{(4C_N)^{n-j-\ell}}{\ell!}(tD_N)^{\ell}\right).
\end{align}
The same argument shows that for all $1\leq j<i\leq n$,
\begin{align}
\nonumber
& \big\| f_{ij}^{(n)}(\cdot,w,z,x;t) \big\|_{\cP_0^{\otimes (n-2)}}^2 \leq \\
\nonumber
& \quad \left( \int_{0<r<\theta<t} p_{t-\theta}(x-w)p_{\theta-r}(w-z)drd\theta\right)^2 \frac{(2t)^{n-2}}{(j-1)!(i-j-1)!(n-i)!} \\
\label{f-ijn-2}
& \quad \left( \sum_{\ell=0}^{j-1} \frac{C_N^{j-1-\ell}}{\ell!} (t D_N)^{\ell}\right) \left( \sum_{\ell=0}^{i-j-1} \frac{(4C_N)^{i-j-1-\ell}}{\ell!}(tD_N)^{\ell}\right) \left( \sum_{\ell=0}^{n-i}\frac{(4C_N)^{n-i-\ell}}{\ell!}(tD_N)^{\ell}\right).
\end{align}

 We now consider two cases.

{\em Case 1.} Suppose that there exists $N_0>0$ such that $C_{N_0}=0$. Then $C_N=0$ for all $N\geq N_0$. Let $N \geq N_0$ be arbitrary. Then the only non-zero terms in the 3 sums appearing on the right-hand-side of \eqref{f-ijn-1} are those corresponding to $\ell=i-1$ (for the first sum), $\ell=j-i-1$ (for the second sum), respectively $\ell=n-j$ (for the third sum). Therefore, \eqref{f-ijn-1} becomes:
\begin{align*}
\big\| f_{ij}^{ (n)}(\cdot,w,z,x;t) \big\|_{\cP_0^{\otimes (n-2)}}^2 \leq f_2^2(w,z,x;t) \frac{(2t^2D_N)^{n-2}}{[(i-1)!(j-i-1)!(n-j)!]^2}.
\end{align*}
Taking the sum for all $1\leq i<j\leq n$, and using the fact that
$\sum_{1\leq i <j \leq n} \binom{n-2}{i-1,j-i-1,n-j}^2 \leq \left(\sum_{1\leq i <j \leq n} \binom{n-2}{i-1,j-i-1,n-j}\right)^2=9^{n-2}$, we obtain:
\begin{align*}
\sum_{1\leq i<j \leq n}\big\| f_{ij}^{ (n)}(\cdot,w,z,x;t) \big\|_{\cP_0^{\otimes (n-2)}}^2 \leq
f_2^2(w,z,x;t) \frac{(18t^2D_N)^{n-2}}{[(n-2)!]^2}.
\end{align*}
For the sum for $j<i$, we use \eqref{f-ijn-2} and obtain that:
\begin{align*}
\sum_{1\leq j<i \leq n}\big\| f_{ij}^{ (n)}(\cdot,w,z,x;t) \big\|_{\cP_0^{\otimes (n-2)}}^2 \leq f_2^2(z,w,x;t) \frac{(18t^2D_N)^{n-2}}{[(n-2)!]^2}.
\end{align*}
Coming back to \eqref{bound-Bn} and using the fact that $n(n-1)\leq 4^{n-2}$, we get:
\[
\bE|B_n(w,z,x;t)|^2 \leq \Big( f_2^2(w,z,x;t)+f_2^2(z,w,x;t) \Big) \frac{(72t^2D_N)^{n-2}}{(n-2)!}.
\]
This proves that $\sum_{n\geq 2}B_n(w,z,x;t)$ converges in $L^2(\Omega)$. Moreover,
\[
\|D_{w,z}^2 u(t,x)\|_2 \leq \sum_{n\geq 2}\|B_n(w,z,x;t)\|_2 \leq C\widetilde{f}_2(w,z,x;t)\exp(C t^2 D_N).
\]
For higher moments (of order $p>2$), we use hypercontractivity to obtain \eqref{bound-D2}.

\medskip

{\em Case 2.} Suppose that $C_N>0$ for all $N>0$. In this case, bounding the three sums appearing on the right-hand-side of \eqref{f-ijn-1}, respectively by $C_{N}^{i-1} e^{C_N^{-1}tD_N}$, $(4C_{N})^{j-i-1} e^{(4C_N)^{-1}tD_N}$ and $(4C_{N})^{n-j} e^{(4C_N)^{-1}tD_N}$, we obtain that for all $i<j$,
\begin{align*}
\big\| f_{ij}^{ (n)}(\cdot,w,z,x;t) \big\|_{\cP_0^{\otimes (n-2)}}^2 \leq f_2^2(w,z,x;t) \frac{(2t)^{n-2}}{(i-1)!(j-i-1)!(n-j)!}(4C_N)^{n-2} e^{2tC_{N}^{-1}D_N}.
\end{align*}
Taking the sum over all $1\leq i<j\leq n$ and using the fact that
$\sum_{1\leq i<j\leq n}\binom{n-2}{i-1,j-i-1,n-j}=3^{n-2}$, we get:
\[
\sum_{1\leq i<j \leq n}\big\| f_{ij}^{ (n)}(\cdot,w,z,x;t) \big\|_{\cP_0^{\otimes (n-2)}}^2 \leq f_2^2(w,z,x;t) \frac{(24 tC_N)^{n-2}}{(n-2)!}e^{2t C_{N}^{-1}D_N}.
\]
Similarly,
\[
\sum_{1\leq j<i \leq n}\big\| f_{ij}^{ (n)}(\cdot,w,z,x;t) \big\|_{\cP_0^{\otimes (n-2)}}^2 \leq f_2^2(z,w,x;t) \frac{(24 tC_N)^{n-2}}{(n-2)!}e^{2t C_{N}^{-1}D_N}.
\]
Coming back to \eqref{bound-Bn} and using again the fact that $n(n-1) \leq 4^{n-2}$, we see that
\[
\bE|B_n(w,z,x;t)|^2 \leq \Big(f_2^2(w,z,x;t) + f_2^2(z,w,x;t) \Big)(96 tC_N)^{n-2}e^{2t C_{N}^{-1}D_N}.
\]
Since $\lim_{N \to \infty}C_N=0$ (by (D)), we can choose $N=N_T$ large enough such that
$$ 96T C_N<1.$$
Then, for any $t \in [0,T]$, $\sum_{n\geq 2}B_n(w,z,x;t)$ converges in $L^2(\Omega)$, and
\[
\|D_{w,z}^2 u(t,x)\|_2 \leq \sum_{n\geq 2}\|B_n(w,z,x;t)\|_2 \leq 2\widetilde{f}_2(w,z,x;t)\frac{1}{1-96T C_N}e^{2T C_{N}^{-1}D_N}.
\]
For moments of order $p>2$, we obtain relation \eqref{bound-D2} by hypercontractivity, as usually.
\end{proof}

\begin{remark}[Higher order Malliavin derivatives]
{\rm Under the hypotheses of Theorem \ref{main-th3}, using the same argument as above, it can be proved that for any $m\in \bN$, and for any $y_1,\ldots,y_m \in \bR^d$,
we have the following chaos expansion:
\[
D_{y_1,\ldots,y_m}^m u(t,x)
	=\sum_{n\geq m} n(n-1)\ldots (n-m+1) I_{n-m}(\widetilde{f}_n(\cdot,y_1,\ldots,y_m,x;t)).
\]
This series converges in $L^2(\Omega)$. Moreover, for any $p \geq 2$, $T>0$, $t\in [0,T]$ and $x \in \bR^d$,
\[
	\|D_{y_1,\ldots,y_m}^m u(t,x) \|_p \leq C_T \widetilde{f}_m(y_1,\ldots,y_m,x;t),
\]
where $C_T>0$ is a constant that depends on $(T,p,\gamma,m)$.
}
\end{remark}

\subsection{Proof of Theorem \ref{main} under Assumption A}

Since $\gamma$ is integrable, $\mu$ has a continuous density $g$, which is bounded by $\|\gamma\|_{L^1(\bR^d)}$.

\medskip

(i) (Limiting covariance) We use the same argument as in the proof of Theorem 1.3.(i) of \cite{BY} (for the wave equation), with some minor differences which we include below. It is enough to prove that
\[
\int_{\bR^d}\rho_{t,s}(z)dz<\infty.
\]
Let $\alpha_n(z;t,s)=(n!)^2\langle \widetilde{f}_n(\cdot,z;t),
\widetilde{f}_n(\cdot,0;s) \rangle_{\cP_0^{\otimes n}}$.
Using the fact that $\int_{\bR^d}p_t(x)dx=1$ and $|\cF p_t(\xi)|^2=e^{-t|\xi|^2} \leq 1$, we obtain that for any $\e>0$,
\begin{align*}
&\int_{\bR^d}\alpha_{n}(z;t,s)e^{-\frac{\e|z|^2}{2}}dz\\
& \quad \leq n! t^2 \|\gamma\|_{L^1(\bR^d)} \int_{T_n(t)} \int_{(\bR^d)^{n-1}} \prod_{j=1}^{n-1} |\cF p_{t_{j+1}-t_j}(\xi_1+\ldots+\xi_j)|^2 \mu(d\xi_1)\ldots \mu(d\xi_n) d\pmb{t} \\
& \quad = n! t^2 \|\gamma\|_{L^1(\bR^d)} \int_0^t J_{n-1}(t_n)dt_n \leq n! t^3 \|\gamma\|_{L^1(\bR^d)} h_{n-1}(t),
\end{align*}
where for the last inequality we used \eqref{Jn-hn} and the fact that $h_{n-1}$ is non-decreasing. Therefore, by Lemma \ref{Lem-3.8},
\[
\int_{\bR^d}\rho_{t,s}(z)dz=\sum_{n\geq 1}\frac{1}{n!}\int_{\bR^d}\alpha_n(z;t,s)dz \leq  t^3 \|\gamma\|_{L^1(\bR^d)} \sum_{n\geq 1}h_{n-1}(t)<\infty.
\]

\medskip

(ii) (QCLT) We apply a version of Proposition 1.8 of \cite{BNQZ} for the time-independent noise, and we use the key estimates for $Du$ and $D^2u$ given by \eqref{D-bound} and \eqref{bound-D2}. We omit the details, as they are very similar to the case of the wave equation (see the proof of Theorem 1.3.(ii) of \cite{BY}).

\medskip

(iii) (FCLT)
{\em Step 1.} (tightness) We prove that for any $R>0$, $p\geq 2$, $0<s<t<T$ and $\alpha \in (0,1/2)$,
\begin{equation}
\label{tight1}
\|F_R(t)-F_R(s)\|_p \leq CR^{d/2} \left( (t-s)^{\alpha} + (t-s)^{1/2} \right),
\end{equation}
where $C>0$ is a constant that depends on $T,\gamma,p,\alpha,d$. From this, it will follow that $F_R=\{F_R(t)\}_{t\geq 0}$ has a continuous modification (denoted also by $F_R$), by Kolmogorov's continuity theorem, and $\{F_R\}_{R >0}$ is tight in $C[0,T]$, by Theorem 12.3 of \cite{billingsley68}.

Using the chaos expansion, we can write
$F_R(t)-F_R(s)= \sum_{n\geq 1} I_n (g_{n,R}(\cdot;t,s))$,
where
\begin{align} \label{eq-decomposition-1}
	& g_{n,R}(x_1,\ldots,x_n;t,s)
	= \int_{B_R} \big( f_n(x_1,\ldots,x_n,x;t) - f_n(x_1,\ldots,x_n,x;s) \big)dx \nonumber \\
	=& \int_{B_R} \int_{T_n(s)} \prod_{j=1}^{n-1} p_{t_{j+1}-t_j} (x_{j+1} - x_j) \big( p_{t-t_n}(x-x_n) - p_{s-t_n}(x-x_n) \big) d\pmb{t} dx \nonumber \\
	&+ \int_{B_R} \int_{T_n(t)} \prod_{j=1}^{n-1} p_{t_{j+1}-t_j} (x_{j+1} - x_j) p_{t-t_n}(x-x_n) \mathbf{1}_{[s,t]}(t_n) d\pmb{t} dx \nonumber \\
	:=& g_{n,R}^{(1)}(x_1,\ldots,x_n;t,s) + g_{n,R}^{(2)}(x_1,\ldots,x_n;t,s).
\end{align}


We first estimate $\left\| \widetilde{g}^{(1)}_{n,R} (\cdot;t,s) \right\|_{\cP_0^{\otimes n}}$. By Lemma 3.1 of \cite{CH2019}, for any $\alpha \in (0,1/2)$, 
\[
	\left| p_{t-t_n}(x-x_n) - p_{s-t_n}(x-x_n) \right|
	\le C_{\alpha} (t-s)^{\alpha} (s-t_n)^{-\alpha} p_{4(t-t_n)} (x-x_n),
\]
where $C_{\alpha}>0$ depends on $\alpha$.
Then $\left| g_{n,R}^{(1)}(x_1,\ldots,x_n;t,s) \right| \le (t-s)^{\alpha} g_{n,R}'(x_1,\ldots,x_n;t,s)$, where
\begin{align*}
	& g_{n,R}'(x_1,\ldots,x_n;t,s)
	= C_{\alpha} \int_{B_R} \int_{T_n(s)} \prod_{j=1}^{n-1} p_{t_{j+1}-t_j} (x_{j+1} - x_j) (s-t_n)^{-\alpha} p_{4(t-t_n)} (x-x_n) d\pmb{t} dx.
\end{align*}
The spatial Fourier transform of $g_{n,R}'$ is:
\begin{align*}
	&\cF g_{n,R}' (\cdot;t,s) (\xi_1,\ldots,\xi_n) \\
	=& C_{\alpha} \int_{B_R} e^{-i(\xi_1+\ldots+\xi_n) \cdot x} \int_{T_n(s)} (s-t_n)^{-\alpha} \prod_{j=1}^{n-1} \cF p_{t_{j+1}-t_j} (\xi_1 + \ldots + \xi_j) \cF p_{4(t-t_n)} (\xi_1 + \ldots + \xi_n) d\pmb{t} dx \\
	=& C_{\alpha} \cF \mathbf{1}_{B_R} (\xi_1+\ldots+\xi_n) \int_{T_n(s)} (s-t_n)^{-\alpha} \prod_{j=1}^{n-1} \cF p_{t_{j+1}-t_j} (\xi_1 + \ldots + \xi_j) \cF p_{4(t-t_n)} (\xi_1 + \ldots + \xi_n) d\pmb{t}.
\end{align*}
By the definition of the norm $\|\cdot\|_{\cP_0^{\otimes n}}$ and the non-negativity of $\gamma$, we have
\begin{align*}
	& n! \left\| \widetilde{g}^{(1)}_{n,R} (\cdot;t,s) \right\|_{\cP_0^{\otimes n}}^2 =n! \langle  g^{(1)}_{n,R} (\cdot;t,s), \widetilde{g}^{(1)}_{n,R} (\cdot;t,s) \rangle_{\cP_0^{\otimes n}}^2 \nonumber \\
	=& n! \int_{(\bR^{d})^n}\int_{(\bR^{d})^n}  g^{(1)}_{n,R} (x_1, \ldots, x_n ;t,s) \widetilde{g}^{(1)}_{n,R} (y_1, \ldots, y_n ;t,s) \prod_{i=1}^n\gamma(x_i-y_i) d\pmb{x} d\pmb{y}\nonumber \\
	\le& n! (t-s)^{2\alpha} \int_{(\bR^{d})^n}  \int_{(\bR^{d})^n}  g'_{n,R} (x_1, \ldots, x_n ;t,s) \widetilde{g'}_{n,R} (y_1, \ldots, y_n ;t,s) \prod_{i=1}^n\gamma(x_i-y_i) d\pmb{x} d\pmb{y}\\
	=& n! (t-s)^{2\alpha} \int_{(\bR^{d})^n} \cF g_{n,R}' (\cdot;t,s) (\xi_1,\ldots,\xi_n) \overline{\cF \widetilde{g'}_{n,R} (\cdot;t,s) (\xi_1,\ldots,\xi_n)} \mu(d\xi_1) \ldots \mu(d\xi_n) \nonumber \\
	=& (t-s)^{2\alpha} \sum_{\rho \in S_n} \int_{(\bR^{d})^n} \cF g_{n,R}' (\cdot;t,s) (\xi_1,\ldots,\xi_n) \overline{\cF g'_{n,R} (\cdot;t,s) (\xi_{\rho(1)},\ldots,\xi_{\rho(n)})} \mu(d\xi_1) \ldots \mu(d\xi_n) \nonumber \\
	=& C_{\alpha}^2 (t-s)^{2\alpha} \sum_{\rho \in S_n} \int_{T_n(s)^2} \int_{(\bR^{d})^n}
	  (s-t_n)^{-\alpha} \prod_{j=1}^{n-1} \cF p_{t_{j+1}-t_j} (\xi_1 + \ldots + \xi_j)  \cF p_{4(t-t_n)} (\xi_1 + \ldots + \xi_n)    \nonumber  \\
&  (s-s_n)^{-\alpha}  \prod_{j=1}^{n-1}\cF p_{s_{j+1}-s_j} (\xi_{\rho(1)} + \ldots + \xi_{\rho(j)})\cF p_{4(t-s_n)} (\xi_{\rho(1)} + \ldots + \xi_{\rho(n)}) \nonumber \\
&
\left| \cF \mathbf{1}_{B_R} (\xi_1+\ldots+\xi_n) \right|^2  \mu(d\xi_1) \ldots \mu(d\xi_n)d\pmb{t} d\pmb{s} .
\nonumber
\end{align*}
By applying Lemma A.1 of \cite{BY}, we obtain:
\begin{align*}
	& n! \left\| \widetilde{g}^{(1)}_{n,R} (\cdot;t,s) \right\|_{\cP_0^{\otimes n}}^2 \leq C_{\alpha}^2 (t-s)^{2\alpha} s^n \int_{T_n(s)} (s-t_n)^{-2\alpha} \int_{\bR^{nd}} \prod_{j=1}^{n-1} \left| \cF p_{t_{j+1}-t_j} (\xi_1 + \ldots + \xi_j) \right|^2 \nonumber \\
& \quad \quad \quad
\left| \cF \left( \mathbf{1}_{B_R} * p_{4(t-t_n)} \right) (\xi_1+\ldots+\xi_n) \right|^2 \mu(d\xi_1) \ldots \mu(d\xi_n) d\pmb{t}.
\nonumber
\end{align*}

Using the maximal principle \eqref{max-principle}, we have:
\begin{align} \label{ineq-norm-(1)}
	& n! \left\| \widetilde{g}^{(1)}_{n,R} (\cdot;t,s) \right\|_{\cP_0^{\otimes n}}^2
	\le C_{\alpha}^2 (t-s)^{2\alpha} s^n \int_{T_n(s)} (s-t_n)^{-2\alpha}  \int_{(\bR^d)^n} \prod_{j=1}^{n-1} \left| \cF p_{t_{j+1}-t_j} (\xi_j) \right|^2
\nonumber \\
	& \times \left| \cF \mathbf{1}_{B_R} (\xi_n) \right|^2 \left| \cF p_{4(t-t_n)} (\xi_n) \right|^2 \mu(d\xi_1) \ldots \mu(d\xi_n) d\pmb{t}
\nonumber \\
	=& C_{\alpha}^2 (t-s)^{2\alpha} s^n \int_0^s (s-t_n)^{-2\alpha} h_{n-1}(t_n) \int_{\bR^d} \left| \cF \mathbf{1}_{B_R} (\xi_n) \right|^2 \left| \cF p_{4(t-t_n)} (\xi_n) \right|^2 \mu(d\xi_n)dt_n.
\end{align}
For the integral with respect to $d\xi_n$, using the fact that $|\cF p_t(\xi)|  \le 1$, we have
\begin{align}
\nonumber
	& \int_{\bR^d} \left| \cF \mathbf{1}_{B_R} (\xi_n) \right|^2 \left| \cF p_{4(t-t_n)} (\xi_n) \right|^2 \mu(d\xi_n)
	\le \int_{\bR^d} \left| \cF \mathbf{1}_{B_R} (\xi_n) \right|^2 \mu(d\xi_n) \\
\label{ineq-BR}
	& \quad \quad \quad =\int_{(\bR^d)^2} \mathbf{1}_{B_R}(x) \mathbf{1}_{B_R}(y) \gamma(x-y) dxdy
	\le C_{\gamma,d} R^d,
\end{align}
where $C_{\gamma,d} = \omega_d \|\gamma\|_{L^1(\bR^d)}$. For the integral with respect to $dt_n$, since $h_{n-1}$ is a non-decreasing function, we have
\begin{align} \label{ineq-integral of h}
	\int_0^s (s-t_n)^{-2\alpha} h_{n-1}(t_n) dt_n
	\le h_{n-1}(s) \int_0^s (s-t_n)^{-2\alpha} dt_n
	= h_{n-1}(s) \dfrac{s^{1-2\alpha}}{1-2\alpha}.
\end{align}
Hence,
\begin{align}
\label{ineq-norm-g(1)}
	n! \left\| \widetilde{g}^{(1)}_{n,R} (\cdot;t,s) \right\|_{\cP_0^{\otimes n}}^2
	\leq \dfrac{C_{\alpha}^2 C_{\gamma,d}}{1-2\alpha} R^d (t-s)^{2\alpha} s^{n+1-2\alpha} h_{n-1}(s).
\end{align}

Next, we estimate $\left\| \widetilde{g}^{(2)}_{n,R} (\cdot;t,s) \right\|_{\cP_0^{\otimes n}}$.
Note that the spatial Fourier transform of $g^{(2)}_{n,R}$ is
\begin{align*}
	\cF g^{(2)}_{n,R} (\cdot;t,s) (\xi_1, \ldots, \xi_n)
	= \cF \mathbf{1}_{B_R} (\xi_1+\ldots+\xi_n) \int_{T_n(t)} \prod_{j=1}^{n} \cF p_{t_{j+1}-t_j} (\xi_1 + \ldots + \xi_j) 1_{[s,t]}(t_n) d\pmb{t},
\end{align*}
where $t_{n+1}=t$.
Expressing the $\cP_0^{\otimes n}$-inner product in Fourier mode, and using Lemma A.1 of \cite{BY}, we have:
\begin{align*}
	& n! \left\| \widetilde{g}^{(2)}_{n,R} (\cdot;t,s) \right\|_{\cP_0^{\otimes n}}^2 = n! \langle g^{(2)}_{n,R} (\cdot;t,s) , \widetilde{g}^{(2)}_{n,R} (\cdot;t,s) \rangle_{\cP_0^{\otimes n}} \nonumber \\
	=& \sum_{\rho \in S_n} \int_{(\bR^{d})^n} \cF g_{n,R}^{(2)} (\cdot;t,s) (\xi_1,\ldots,\xi_n) \overline{\cF g^{(2)}_{n,R} (\cdot;t,s) (\xi_{\rho(1)},\ldots,\xi_{\rho(n)})} \mu(d\xi_1) \ldots \mu(d\xi_n) \nonumber \\
	=& \sum_{\rho \in S_n} \int_{T_n(t)}  \int_{T_n(t)}  \int_{(\bR^{d})^n}
\prod_{j=1}^{n} \cF p_{t_{j+1}-t_j} (\xi_1 + \ldots + \xi_j)
\prod_{j=1}^{n} \cF p_{s_{j+1}-s_j} (\xi_{\rho(1)} + \ldots + \xi_{\rho(j)})
\nonumber \\
	& \times 1_{[s,t]} (t_n) 1_{[s,t]}(s_n)
\left| \cF \mathbf{1}_{B_R} (\xi_1+\ldots+\xi_n) \right|^2 \mu(d\xi_1) \ldots \mu(d\xi_n)
d\pmb{s}   d\pmb{t}  \nonumber \\
	\le& t^n \int_{T_n(t)} 1_{[s,t]}(t_n) \ \int_{(\bR^{d})^n} \prod_{j=1}^{n-1} \left| \cF p_{t_{j+1}-t_j} (\xi_1 + \ldots + \xi_j) \right|^2
\nonumber \\
& \times \left| \cF (\mathbf{1}_{B_R}* p_{t-t_n}) (\xi_1+\ldots+\xi_n) \right|^2 \mu(d\xi_1) \ldots \mu(d\xi_n)  d\pmb{t}.
\end{align*}
Using the maximal principle \eqref{max-principle} and an inequality similar to \eqref{ineq-BR} but for $p_{t-t_n}$ instead of $p_{4(t-t_n)}$, we have:
\begin{align}
	& n! \left\| \widetilde{g}^{(2)}_{n,R} (\cdot;t,s) \right\|_{\cP_0^{\otimes n}}^2 \nonumber \\
	\le& t^n \int_{T_n(t)} 1_{[s,t]} (t_n) \int_{(\bR^{d})^n}
	\prod_{j=1}^{n-1} \left| \cF p_{t_{j+1}-t_j} (\xi_j) \right|^2 \left| \cF \left( \mathbf{1}_{B_R} * p_{t-t_n} \right) (\xi_n) \right|^2 \mu(d\xi_1) \ldots \mu(d\xi_n)  d\pmb{t} \nonumber \\
\label{ineq-norm-(2)}
	=& t^n \int_s^t h_{n-1}(t_n) \left(\int_{\bR^d} \left| \cF \mathbf{1}_{B_R} (\xi_n) \right|^2 \left| \cF p_{t-t_n} (\xi_n) \right|^2 \mu(d\xi_n)\right) dt_n \\
\nonumber
\leq & t^n C_{\gamma,d} R^d \int_s^t h_{n-1}(t_n) dt_n.
\end{align}
Using the fact that $h_{n-1}$ is a non-decreasing function, we obtain:
\begin{align}
\label{ineq-norm-g(2)}
	n! \left\| \widetilde{g}^{(2)}_{n,R} (\cdot;t,s) \right\|_{\cP_0^{\otimes n}}^2
	\le C_{\gamma,d} t^n (t-s) h_{n-1}(t) R^d.
\end{align}

By the hypercontractivity, \eqref{ineq-norm-g(1)}, \eqref{ineq-norm-g(2)} and Lemma \ref{Lem-3.8}, we have
\begin{align*}
	\|F_R(t)-F_R(s)\|_p
	\le& \sum_{n\geq 1} (p-1)^{n/2} (n!)^{1/2} \Big(\left\| \widetilde{g}^{(1)}_{n,R} (\cdot;t,s) \right\|_{\cP_0^{\otimes n}}
	+ \left\| \widetilde{g}^{(2)}_{n,R} (\cdot;t,s) \right\|_{\cP_0^{\otimes n}} \Big) \\
	\le& C R^{d/2} \sum_{n\geq 1} (p-1)^{n/2}  \sqrt{h_{n-1}(t)}\left[\sqrt{ s^{n+1-2\alpha} (t-s)^{2\alpha} }+
	  \sqrt{t^n (t-s) } \right] \\
	\le& C R^{d/2} \left( (t-s)^{\alpha} + (t-s)^{1/2} \right),
\end{align*}
where $C>0$ depends on $p, t, \gamma, \alpha, d$ and is increasing in $t$. This proves \eqref{tight1}.

{\em Step 2.} (finite-dimensional convergence)
Recall that $Q_R(t)=R^{-d/2}F_R(t)$. The same argument as in the proof of Theorem 1.3 (iii) (Step 2) of \cite{BY} (for wave equation) shows that for any $m \in \bN_+$, $0\leq t_1<\ldots<t_m\leq T$,
\begin{equation*}
	(Q_R(t_1),\ldots,Q_R(t_m)) \stackrel{d}{\to} (\cG(t_1),\ldots,\cG(t_m)).
\end{equation*}

\subsection{Proof of Theorem \ref{main} under Assumption B} \label{sec:CLT Riesz}

Under Assumption B, $\Gamma$ is the Fourier transform of the measure $\mu(d\xi) = c_{d,\beta} |\xi|^{\beta-d} d\xi$.

(i) (Limiting Covariance) We use a similar argument as in the proof of Theorem 1.4.(i) of \cite{BY} (for the wave equation). We sketch this below.

We have the chaos expansion
$F_R(t) = \sum_{n\geq 1} {\bf J}_{n,R} (t)$, with
\begin{align}
\label{def-JnR}
	{\bf J}_{n,R} (t) = I_n(g_{n,R}(\cdot;t))  \quad \mbox{and} \quad
g_{n,R}(\cdot;t)=\int_{B_R} f_n(\cdot,x;t) dx.
\end{align}

We show that only the first-chaos term contributes to the limit. For this term, we use the fact that
$|\cF p_t(\xi/R)| = \exp (-\frac{t|\xi|^2}{2R^2})$ converges to 1 as $R \to \infty$, and is bounded by $1$. Hence,
\begin{align} \label{eq-E[I1I1]}
	\bE [ {\bf J}_{1,R} (t) {\bf J}_{1,R} (s) ]
	\sim R^{2d-\beta} ts \int_{B_1^2} |x-x'|^{-\beta} dxdx' \quad \mbox{as} \quad R \to \infty.
\end{align}
Next, we show that the chaos terms of order $n\geq 2$ are negligible.
By (64) of \cite{BY},
\begin{align*}
	& \bE [ {\bf J}_{n,R}^2(t)] \\
	&\le t^n c_{d,\beta}^n R^{2d-\beta} \int_{T_n(t)} dt_1 \ldots dt_n \int_{(\bR^d)^{n-1}}  d\eta_1 \ldots d\eta_{n-1}\prod_{j=1}^{n-1} |\eta_j-\eta_{j-1}|^{-(d-\beta)} \prod_{j=1}^{n-1} |\cF p_{t_{j+1}-t_j}(\eta_j)|^2 \\
	& \quad \int_{\bR^d} \left( \int_{B_1^2}
	e^{-i \eta_n \cdot (x-x')} dxdx'\right) |\cF p_{t-t_n}(\eta_n/R)|^2 |\eta_n - \eta_{n-1}R|^{\beta-d} d\eta_n,
\end{align*}
where $\eta_0=0$. The right-hand side coincide with the integral on the right-hand side of relation (3.35) in \cite{NZ20EJP}. Thus, by the computation in \cite{NZ20EJP}, we have:
\begin{align} \label{eq-E[InIn]}
	\sum_{n\geq 2} \bE [ {\bf J}_{n,R}^2 (t) ]
	= o(R^{2d-\beta}) \quad {\rm as} \quad R \to \infty.
\end{align}
The conclusion follows by \eqref{eq-E[I1I1]} and \eqref{eq-E[InIn]}.

\medskip

(ii) (QCLT) The argument is similar to the Section 3.1.2 of \cite{NXZ} for the heat equation with colored noise in time, but the details are slightly different.

We apply a version of Proposition 1.8 of \cite{BNQZ} for the time-independent noise. We obtain:
\[
d_{TV}\left( \frac{F_R(t)}{\sigma_R(t)},Z\right)\leq \frac{4}{\sigma_R^2(t)}\sqrt{\cA}
\]
where $Z\sim N(0,1)$ and
\begin{align*}
\cA&=\int_{(\bR^d)^6} \|D_{z,w}^2 F_R(t)\|_4
 \|D_{y,w'}^2 F_R(t)\|_4  \|D_{z'} F_R(t)\|_4  \|D_{y'} F_R(t)\|_4\\
 & \qquad \qquad \qquad \gamma(y-y') \gamma(z-z')\gamma(w-w')dydy' dz dz' dw dw'.
\end{align*}
Since $\sigma_{R}^2(t) \sim K(t,t) R^{2d-\beta}$ (by part (i)), it is enough to prove that
\[
\cA \leq C_t R^{4d-3\beta},
\]
where $C_t>0$ is a constant depending on $t$. Note that $D_{z}F_R(t)=\int_{B_R}u(t,x)dx$. By Minkowski's inequality and \eqref{D-bound},
\[
\|D_z F_R(t)\|_4 \leq \int_{B_R}\|D_z u(t,x)\|_4 dx \leq C_t \int_{B_R} f_1(z,x;t)dx.
\]
Similarly, using \eqref{bound-D2},
\[
\|D_{z,w}^2 F_R(t)\|_4 \leq \int_{B_R}\|D_{z,w}^2 u(t,x)\|_4 dx \leq C_t \int_{B_R} \widetilde{f}_2(z,w,x;t)dx.
\]
Hence
\begin{align*}
\cA & \leq C_t^4 \int_{B_R^4} \int_{(\bR^d)^6}
\widetilde{f}_2(z,w,x_1;t) \widetilde{f}_2(y,w',x_2;t) f_1(z',x_3;t) f_1(y',x_4,x;t) \\
& \qquad \gamma(y-y') \gamma(z-z')\gamma(w-w') dydy' dzdz' dw dw' dx_1 dx_2 dx_3 dx_4 =:C_t^4 \sum_{i=1}^4 \cA_i.
\end{align*}
We treat only $\cA_1$, the other terms being similar. We have:
\begin{align*}
& \cA_1 =\int_{B_R^4} \int_{(\bR^d)^6}\left( \int_{0<\theta<r<t} p_{t-r}(x_1-z) p_{r-\theta}(z-w)dr d\theta\right) \\
& \ \left( \int_{0<\theta'<s<t} p_{t-s}(x_2-y) p_{s-\theta'}(y-w')ds d\theta'\right) \left( \int_0^t p_{t-r'}(x_3-z') dr'\right)
\left( \int_0^t p_{t-s'}(x_4-y') ds'\right) \\
& \quad  \gamma(y-y') \gamma(z-z')\gamma(w-w') dydy' dzdz' dw dw' dx_1 dx_2 dx_3 dx_4.
\end{align*}
We notice that the integral on $B_R^4 \times (\bR^d)^6$ coincides with the integral (on the same domain) which appears on the right-hand side of $\cA^*$ in Section 3.1.2 of \cite{NXZ}. Therefore, the remaining part of the argument is the same as in \cite{NXZ}.

\medskip

(iii) (CLT)
{\em Step 1.} (tightness) We prove that for any $R>0$, $p\geq 2$, $0<s<t<T$ and $\alpha \in (0,\frac{1}{2})$,
\begin{equation}
\label{Riesz-tight}
\|F_R(t)-F_R(s)\|_p \leq C R^{d-\beta/2} \left( (t-s)^{\alpha} + (t-s)^{1/2} \right),
\end{equation}
where $C>0$ is a constant depending on $T,\gamma,p,\alpha,d$.
The conclusion will follow by Kolmogorov's continuity theorem and Theorem 12.3 of \cite{billingsley68}.

Note that formula \eqref{eq-decomposition-1} still holds. For the term $\left\| \widetilde{g}^{(1)}_{n,R} (\cdot;t,s) \right\|_{\cP_0^{\otimes n}}$, the estimate \eqref{ineq-norm-(1)} is still valid. For the integral with respect to $d\xi_n$ in \eqref{ineq-norm-(1)}, noting that $|\cF p_t(\xi)| = \exp(-t|\xi|^2/2) \le 1$ for all $t > 0$ and $\xi \in \bR^d$, we have
\begin{align}
\label{BR-Riesz}
	\int_{\bR^d} \left| \cF \mathbf{1}_{B_R} (\xi_n) \right|^2 \left| \cF p_{4(t-t_n)} (\xi_n) \right|^2 \mu(d\xi_n)
	\le \int_{\bR^d} \left| \cF \mathbf{1}_{B_R} (\xi_n) \right|^2 \mu(d\xi_n)
	= C_{\gamma,d}' R^{2d-\beta},
\end{align}
where $C_{\gamma,d}'$ is a positive constant that only depends on $\gamma, d$, and we use relation (68) of \cite{BY} in the equality. Hence, together with \eqref{ineq-integral of h}, we can write \eqref{ineq-norm-(1)} as
\begin{align} \label{eq-norm-g(1)-Riesz}
	n! \left\| \widetilde{g}^{(1)}_{n,R} (\cdot;t,s) \right\|_{\cP_0^{\otimes n}}^2
	\le \dfrac{C_{\alpha}^2 C_{\gamma,d}'}{1-2\alpha} R^{2d-\beta} (t-s)^{2\alpha} s^{n+1-2\alpha} h_{n-1}(s).
\end{align}
For the term $\left\| \widetilde{g}^{(2)}_{n,R} (\cdot;t,s) \right\|_{\cP_0^{\otimes n}}$, the estimate \eqref{ineq-norm-(2)} is still valid. For the inner integral, we use an inequality similar to \eqref{BR-Riesz} but for $p_{t-t_n}$ instead of $p_{4(t-t_n)}$.
So, we can write \eqref{ineq-norm-(2)} as
\begin{align} \label{eq-norm-g(2)-Riesz}
	n! \left\| \widetilde{g}^{(2)}_{n,R} (\cdot;t,s) \right\|_{\cP_0^{\otimes n}}^2
	\le C_{\gamma,d}' t^n (t-s) h_{n-1}(t) R^{2d-\beta}.
\end{align}
Relation \eqref{Riesz-tight} follows by hypercontractivity, \eqref{eq-norm-g(1)-Riesz}, \eqref{eq-norm-g(2)-Riesz} and Lemma \ref{Lem-3.8}.

{\em Step 2.} (finite-dimensional convergence)
Recall that $Q_R(t)=R^{-d+\beta/2}F_R(t)$. The same argument (based on the domination of the first chaos) as in the proof of Theorem 1.4 (iii) (Step 2) of \cite{BY} (for wave equation) shows that for any $m \in \bN_+$, $0\leq t_1<\ldots<t_m\leq T$,
\begin{equation*}
	(Q_R(t_1),\ldots,Q_R(t_m)) \stackrel{d}{\to} (\cG(t_1),\ldots,\cG(t_m)).
\end{equation*}


\section{The Rough Case}

In this section, we examine the case of the rough noise. One of the differences compared with the regular case is that the maximal principle \eqref{max-principle} does not hold. We first examine the question of existence of solution, then we give some estimates for the increments of the Malliavin derivative of the solution, and finally we give the proof of Theorem \ref{main} under Assumption C.

\subsection{Existence of solution}

In this subsection, we establish the existence of the solution to \eqref{eq-heat} under Assumption C.

We will use frequently the following properties.
By Stirling's formula, for any $a>0$,
 $\Gamma(an+1) \sim (n!)^a a^{an} a^{1/2} (2\pi n)^{\frac{1-a}{2}}$, and hence,
\begin{equation}
\label{Gamma1}
c_1^n (n!)^a \leq \Gamma(an+1)\leq c_2^n (n!)^a \quad \mbox{for all $n\geq 1$},
\end{equation}
where $c_1,c_2>0$ are constants depending on $a$. In addition,
\begin{equation}
\label{ML-function}
 \sum_{n\geq 0}\frac{x^n}{(n!)^a}\leq c_1 \exp(c_2 x^{1/a}) \quad \mbox{for all $x>0$}.
 \end{equation}

We denote by $c>0$ a constant depending on $H$ that may be different in each of its appearances. We first list some inequalities which we will use frequently below.

From relation (4.16) of \cite{NXZ} (with $H_0=1/2$ and $s=0$) and \eqref{Gamma1}, we have:
\begin{equation}
\label{4-16}
\|f_n(\cdot,t,x)\|_{\cH_0^{\otimes n}}^2=\int_{T_n(t)}\|f_n(\pmb{t_n},\bullet,t,x)\|_{\cP_0^{\otimes n}}^2 d\pmb{t_n} \leq
c^n \frac{t^{nH}}{\Gamma(nH+1)}\leq c^n \frac{t^{nH}}{(n!)^H}.
\end{equation}
Therefore, using Cauchy-Schwarz inequality, we obtain:
\begin{align}
\label{in1'}
	\left(\int_{T_n(t)} \|f_{n}
	(\pmb{t_{n}},\bullet,t,x)\|_{\cP_0^{\otimes n}}d\pmb{t_n}\right)^2   \leq \frac{t^n}{n!} \|f_{n}(\cdot,t,x)\|_{\cH_0^{\otimes n}}^2 \leq
c^n \frac{t^{n(H+1)}}{(n!)^{H+1}}.
\end{align}

By relation (4.19) of \cite{NXZ} (with $H_0=1/2$),
\[
n! \|g_n(\cdot,r,z,t,x)\|_{\cH_0^{\otimes n}}^2 \leq c^n \Gamma((1-H)n+1)(t-r)^{nH}p_{t-r}^2(x-z),
\]
which, due to \eqref{Gamma1}, can be written as:
\begin{equation}
\label{4-19}
\|g_n(\cdot,r,z,t,x)\|_{\cH_0^{\otimes n}}^2 \leq c^n \frac{(t-r)^{nH}}{(n!)^{H}}p_{t-r}^2(x-z).
\end{equation}
Therefore, by Cauchy-Schwarz inequality,
\begin{align}
\label{in2'}
	\left( \int_{\{r<t_1<\ldots<t_{n}<t\}} \|g_{n}
	(\pmb{t_{n}},\bullet,r,z,t,x)\|_{\cP_0^{\otimes n}}d\pmb{t_{n}} \right)^2  
	\leq c^{n} \frac{(t-r)^{n(H+1)}}{(n!)^{H+1}}p_{t-r}^2(x-z).
\end{align}

Note that \eqref{4-16}, \eqref{in1'}, \eqref{4-19} and \eqref{in2'} are the respective analogues of \eqref{norm-fn-white}, \eqref{in1}, \eqref{3-17-NXZ} and \eqref{in2} for the rough case.

\begin{lemma}
\label{Thm-existence-heat-rough}
Suppose Assumption C holds. Then equation \eqref{eq-heat} has a unique solution $u$. Moreover, for any $p \geq 2$, $t>0$ and $x \in \bR$,
\begin{equation}
\label{rough-mom-p}
	\|u(t,x)\|_p\leq c_1 \exp\left(c_2 (p-1)^{\frac{1}{H}} t^{\frac{H+1}{H}}  \right)=:C_1(t),
\end{equation}
where $c_1>0$ and $c_2>0$ are some constants depending on $H$.
\end{lemma}

\begin{proof}
By Minkowski inequality and \eqref{in1'}, we have
\[
\|f_n(\cdot,x;t)\|_{\cP_0^{\otimes n}}=\left\| \int_{T_n(t)} f_n(\pmb{t_n},\bullet,t,x)d\pmb{t_n}\right\|_{\cP_0^{\otimes n}}
\leq \int_{T_n(t)}\| f_n(\pmb{t_n},\bullet,t,x)\|_{\cP_0^{\otimes n}}d\pmb{t_n} \leq \frac{c^{n/2} t^{\frac{n(H+1)}{2}}}{(n!)^{\frac{H+1}{2}}}.
\]
Using the trivial bound $\|\widetilde{f}\|_{\cP_0^{\otimes n}}\leq
\|f\|_{\cP_0^{\otimes n}}$, we have:
\[
\bE|u(t,x)|^2=
\sum_{n\geq 1}n!\|\widetilde{f}_n(\cdot,x;t)\|_{\cP_0^{\otimes n}}^2
\leq \sum_{n\geq 1}n!\|f_n(\cdot,x;t)\|_{\cP_0^{\otimes n}}^2 \leq  \sum_{n\geq 1} c^n \frac{t^{n(H+1)}}{(n!)^H}.
\]
Relation \eqref{rough-mom-p} for $p=2$ follows by \eqref{ML-function}.
For general $p\geq 2$, by hypercontractivity,
\[
\|u(t,x)\|_p \leq \sum_{n\geq 0}(p-1)^{n/2}(n!)^{1/2}\|f_n(\cdot,x;t)\|_{\cP_0^{\otimes n}} \leq \sum_{n\geq 0}(p-1)^{n/2}c^{n/2}\frac{t^{n(H+1)/2}}{(n!)^{H/2}}.
\]
Relation \eqref{rough-mom-p} follows by \eqref{ML-function}.
\end{proof}

\begin{remark}
{\rm Relation \eqref{rough-mom-p} extends Corollary 5.4 of \cite{HLN17} (with $\alpha=2-2H$) to the case of time-independent noise (when $\alpha_0=0$). In Remark 5.3 of \cite{HLN17}, the authors mention that their methods do not cover the case $\alpha_0=0$.
}
\end{remark}

\subsection{Estimates on the Malliavin derivatives}
\label{section-increments}

In this section, we provide some estimates for the increments of $Du(t,x)$ and the rectangular increments of $D^2 u(t,x)$, which will be used for the proof of the QCLT below. These estimates are similar to those given by Proposition 4.1 of \cite{NXZ} for the solution of the parabolic Anderson model with colored noise in time.

We recall the notation from \cite{NXZ}: for any $t>0$ and $x,x',x'' \in \bR^d$,
\begin{align*}
\Delta_t(x,x')&=p_t(x+x')-p_t(x)\\
R_t(x,x',x'')&=p_t(x+x'-x'')-p_t(x+x')-p_t(x-x'')+p_t(x).
\end{align*}
The following function is the one given by relation (2.16) of \cite{NXZ} in the case $H_0=1/2$:
\[
N_t(x) = t^{-1/8}|x|^{1/4}  \mathbf{1}_{\{|x| \le \sqrt{t}\}} + \mathbf{1}_{\{|x| > \sqrt{t}\}}, \quad t>0,x \in \bR^d.
\]

We will need some inequalities for the increments of $f_n$ and $g_n$ which we borrow from \cite{NXZ}. We list them below.

Using the Cauchy-Schwarz inequality, followed by relations (4.18), (4.20)-(4.22) of \cite{NXZ} (with $H_0 = 1/2$), and the estimate \eqref{Gamma1} for the Gamma function, we obtain:
\begin{align}
\label{ineq-norm-difference f_j}
	& \left( \int_{T_{n}(t)} \| f_{n} (\pmb{t_{n}},\bullet,t,x+x') - f_{n} (\pmb{t_{n}},\bullet,t,x) \|_{\cP_0^{\otimes (n)}} d\pmb{t_{n}} \right)^2 \leq c^{n} \dfrac{t^{n(H+1)}}{(n!)^{(H+1)}}  N_r^2(x').
\end{align}
\begin{align}
\nonumber
	&\left( \int_{\{r<t_1<\ldots<t_{n}<t\}}  \| g_{n}(\pmb{t_{n}},\bullet,r,z+z',t,x) - g_{n}(\pmb{t_{n}},\bullet,r,z,t,x) \|_{\cP_0^{\otimes (n-j)}} d\pmb{t_{n}} \right)^2 \nonumber \\
	\label{ineq-norm-difference g_j}
	& \quad \leq  c^{n} \dfrac{(t-r)^{n(H+1)}}{(n!)^{(H+1)}} \Big( |\Delta_{t-r}(z-x,z')| + p_{t-r}(x-z) N_{t-r}(z') \Big)^2.
\end{align}
\begin{align}
\nonumber
	&\left( \int_{\{r<t_1<\ldots<t_{n}<t\}}  \| g_{n}(\pmb{t_{n}},\bullet,r,z,t,x+x') - g_{n}(\pmb{t_{n}},\bullet,r,z,t,x) \|_{\cP_0^{\otimes (n-j)}} d\pmb{t_{n}} \right)^2 \nonumber \\
	\label{ineq-norm-difference g_j-2}
	& \quad \leq  c^{n} \dfrac{(t-r)^{n(H+1)}}{(n!)^{(H+1)}} \Big( |\Delta_{t-r}(x-z,x')| + p_{t-r}(x-z) N_{t-r}(x') \Big)^2.
\end{align}
\begin{align}
\nonumber
&\left( \int_{\{r<t_1<\ldots<t_{n}<t\}}  \| g_{n}(\pmb{t_{n}},\bullet,r,z+z',t,x+x') - g_{n}(\pmb{t_{n}},\bullet,r,z,t,x+x') \right.\\
\nonumber
& \qquad \qquad \qquad \qquad \quad \left.-
g_{n}(\pmb{t_{n}},\bullet,r,z+z',t,x)+g_{n}(\pmb{t_{n}},\bullet,r,z,t,x) \|_{\cP_0^{\otimes (n-j)}} d\pmb{t_{n}} \right)^2 \\
\nonumber
	& \quad \leq  c^{n} \dfrac{(t-r)^{n(H+1)}}{(n!)^{(H+1)}} \Big( |R_{t-r}(x-z,x',z')| + |\Delta_{t-r}(x-z,x')|N_{t-r}(z') \\
& \label{rectangular-diff}
\quad \quad \quad +|\Delta_{t-r}(z-x,z')|N_{t-r}(x')
+ p_{t-r}(x-z) N_{t-r}(x')N_{t-r}(z') \Big)^2.
\end{align}

The following results can be viewed as extensions of Lemmas 4.2 and 4.3 of \cite{NXZ} to the time-independent case (when formally, $H_0=1$).

\begin{lemma}
Under Assumption C, the series in \eqref{D-series} converges in $L^2(\Omega)$, and
for any $p\geq 2$, $t>0$ and $x,z,z'\in \bR$,
\begin{equation}
\label{D-bound-rough}
\|D_z u(t,x)\|_p \leq C_1(t) \int_0^t p_{t-r}(x-z)dr
\end{equation}
an
\begin{equation}
\label{D-bound-incr-rough}
\|D_{z+z'} u(t,x)-D_{z}u(t,x)\|_p
\leq C_1(t) \int_0^t \Big(|\Delta_{t-r}(z-x,z')| + p_{t-r}(x-z) N_{t-r}(z')  \Big)dr,
\end{equation}
where the constant $C_1(t)$ is the same as in Lemma \ref{Thm-existence-heat-rough}.
\end{lemma}

\begin{proof}
We first prove that series in \eqref{D-series} converges in $L^2(\Omega)$. We argue as in the proof of Theorem \ref{main-th3}.
Formulas \eqref{bound-An} and \eqref{eq-norm-f_j^n-heat} are still valid. On the right hand side of \eqref{eq-norm-f_j^n-heat}, we estimate separately the two integrals, using \eqref{in1'} for the first integral and \eqref{in2'} for the second integral.
It follows that
\begin{align*}
	& \|f_j^{(n)}(\cdot,z,x;t)\|_{\cP_0^{\otimes (n-1)}} \leq c^{\frac{n-1}{2}} \frac{t^{\frac{(n-1)(H+1)}{2}}}{[(j-1)!]^{\frac{H+1}{2}} [(n-j)!]^{\frac{H+1}{2}}}\int_0^t p_{t-r}(x-z) dr.
\end{align*}

Coming back to \eqref{bound-An}, we obtain:
\begin{align}
\nonumber
\bE|A_n(z,x;t)|^2 & \leq n! c^{n-1} t^{(n-1)(H+1)}\left(\int_0^t p_{t-r}(x-z)dr\right)^2 \sum_{j=1}^{n}\frac{1}{[(j-1)!]^{H+1}[(n-j)!]^{H+1}} \\
\label{An-rough}
& \leq
\frac{1}{[(n-1)!]^H} c^{n-1} t^{(n-1)(H+1)} \left(\int_0^t p_{t-r}(x-z)dr\right)^2.
\end{align}
Using the fact that:
\begin{equation}
\label{combin-ineq}
\sum_{j=1}^n \binom{n-1}{j-1}^{H+1}\leq \sum_{j=1}^n \binom{n-1}{j-1}^{2}\leq
\left(\sum_{j=1}^n \binom{n-1}{j-1}\right)^{2}=4^{n-1}.
\end{equation}
From this, we deduce that the series in \eqref{D-series} converges in $L^2(\Omega)$, and
\[
\bE|D_z u(t,x)|^2 =\sum_{n\geq 1}\bE|A_n(z,x;t)|^2 \leq \left(\int_0^t p_{t-r}(x-z)dr\right)^2\sum_{n\geq 1}
\frac{c^{n-1} t^{(n-1)(H+1)}}{[(n-1)!]^H}.
\]
This proves \eqref{D-bound-rough} for $p=2$.
For higher moments, we use hypercontractivity and \eqref{An-rough}.

\medskip

Next, we prove \eqref{D-bound-incr-rough} with a similar argument. We denote $\bar{z}=z+z'$. By the chaos expansion \eqref{D-series}, we have
\begin{align}
\label{Du-Du-series}
	D_{z}u(t,x) - D_{\bar{z}}u(t,x)
	= \sum_{n\geq 1} n I_{n-1} \left( \widetilde{f}_n(\cdot,z,x;t) - \widetilde{f}_n(\cdot,\bar{z},x;t) \right)
	=:\sum_{n\geq 1} A^D_n(z,z',x;t).
\end{align}
The same argument leading to \eqref{bound-An} shows that:
\begin{equation} \label{bound-A^Dn}
	\bE|A^D_n(z,z',x;t)|^2
	\leq n! \sum_{j=1}^{n} \|f_j^{(n)}(\cdot,z,x;t) - f_j^{(n)}(\cdot,\bar{z},x;t) \|_{\cP_0^{\otimes (n-1)}}^2.
\end{equation}
Similar to \eqref{eq-norm-f_j^n-heat}, by triangle inequality and Minkowski's inequality, we have:
 \begin{align} \label{ineq-norm-different f_jn}
	& \|f_j^{(n)}(\cdot,z,x;t) - f_j^{(n)}(\cdot,\bar{z},x;t) \|_{\cP_0^{\otimes (n-1)}} \nonumber \\
	\le& \int_0^t
\left\{\left( \int_{\{0<t_1<\ldots<t_{j-1}<r\}}\| f_{j-1} (\pmb{t_{j-1}},\bullet,r,z) \|_{\cP_0^{\otimes (j-1)}}  d\pmb{t_{j-1}} \right) \right. \nonumber \\
	& \left( \int_{\{r<t_j<\ldots<t_{n-1}<t\}} \| g_{n-j}(\pmb{t_{j:n-1}},\bullet,r,z,t,x) - g_{n-j}(\pmb{t_{j:n-1}},\bullet,r,\bar{z},t,x) \|_{\cP_0^{\otimes (n-j)}} d\pmb{t_{j:n-1}} \right) \nonumber \\
	& + \left( \int_{\{r<t_j<\ldots<t_{n-1}<t\}} \| g_{n-j}(\pmb{t_{j:n-1}},\bullet,r,\bar{z},t,x) \|_{\cP_0^{\otimes (n-j)}} d\pmb{t_{j:n-1}}  \right) \nonumber \\
	& \left. \left( \int_{\{0<t_1<\ldots<t_{j-1}<r\}} \| f_{j-1} (\pmb{t_{j-1}},\bullet,r,z) - f_{j-1} (\pmb{t_{j-1}},\bullet,r,\bar{z}) \|_{\cP_0^{\otimes (j-1)}} d\pmb{t_{j-1}}  \right)\right\}dr.
\end{align}

The four integrals appearing inside the $dr$ integral above are bounded using \eqref{in1'}, \eqref{in2'}, \eqref{ineq-norm-difference f_j} and
\eqref{ineq-norm-difference g_j}. Thus,
\begin{align*}
	&\|f_j^{(n)}(\cdot,z,x;t) - f_j^{(n)}(\cdot,\bar{z},x;t) \|_{\cP_0^{\otimes (n-1)}} \leq c^{n-1} \dfrac{t^{(n-1)(H+1)/2}}{[(j-1)!]^{(H+1)/2} [(n-j)!]^{(H+1)/2}}
\\
& \quad \times \int_0^t \Big( |\Delta_{t-r}(z-x,z')| + p_{t-r}(x-z) N_{t-r}(z')  \Big)dr.
\end{align*}

Returning to \eqref{bound-A^Dn}, and using the combinatorial inequality \eqref{combin-ineq}, we obtain:
\begin{align*}
	\bE|A^D_n(z,z',x;t)|^2
	\le \dfrac{c^{n-1} t^{(n-1)(H+1)}}{[(n-1)!]^H} \left[ \int_0^t \Big(|\Delta_{t-r}(z-x,z')| + p_{t-r}(z-x) N_{t-r}(z') \Big) dr \right]^2.
\end{align*}
Relation \eqref{D-bound-incr-rough} follows from \eqref{ML-function} and hypercontractivity.
\end{proof}

\begin{lemma}
Under Assumption C holds, the series in \eqref{D2-series} converges in $L^2(\Omega)$, and
	for any $p\geq 2$, $t>0$ and $x,z,z',w,w'\in \bR$,
	\begin{equation} \label{D2-bound-rough}
		\|D^2_{w,z} u(t,x)\|_p \leq C_1(t) \widetilde{f}_2(w,z,x;t),
	\end{equation}
	and
\begin{align}
\label{D2-bound-incr-rough}
	&\big\| D^2_{w+w',z+z'} u(t,x) - D^2_{w+w',z} u(t,x) - D_{w,z+z'}u(t,x)+ D_{w,z}u(t,x)  \big\|_p \nonumber \\
	\leq& C_1(t) \left(\int_{0<\theta<r<t} H(\theta,w,w',r,z,z',x;t)d\theta dr+ \int_{0<\theta<r<t} H(r,z,z',\theta,w,w',x;t)dr d\theta\right),
\end{align}
where $H(\theta,w,w',r,z,z',x;t)$ is given by the second factor on the right-hand side of (4.5) of \cite{NXZ} with $(r,s,y,y',z,z')$ replaced by $(\theta,r,z,z',w,w')$, i.e.
\begin{align*}
& H(\theta,w,w',r,z,z',x;t)=p_{t-r}(x-z)N_{\theta}(w')\Big[ |\Delta_{r-\theta}(z-w,z')|+p_{r-\theta}(z-w)N_{r-\theta}(z')\Big]\\
& \quad +p_{t-r}(x-z) \Big[ |R_{r-\theta}(z-w,z',w')| +|\Delta_{r-\theta}(z-w,z')| N_{r-\theta}(w')\\
& \quad \quad \quad +|\Delta_{r-\theta}(w-z,w')| N_{r-\theta}(z')+p_{r-\theta}(z-w) N_{r-\theta}(z') N_{r-\theta}(w')\Big]\\
& \quad +p_{r-\theta}(z+z'-w)N_{\theta}(w')\Big[ |\Delta_{t-r}(z-x,z')|+p_{t-r}(x-z)N_{t-r}(z') \Big]\\
& \quad+\Big[ |\Delta_{r-\theta}(w-z-z',w')|+p_{r-\theta}(z+z'-w)N_{r-\theta}(w') \Big] \\ & \quad \quad \quad \times \Big[ |\Delta_{t-r}(z-x,z')|+p_{t-r}(x-z)N_{t-r}(z')\Big].
\end{align*}
where $C_1(t)$ is the same as in Lemma \ref{Thm-existence-heat-rough}.
\end{lemma}

\begin{proof}
We first show that the series in \eqref{D2-series} converges in $L^2(\Omega)$. We proceed as in the proof of Theorem \ref{Thm-D2-bound}.
Note that relations \eqref{bound-Bn} and \eqref{ineq-fijn} still hold. We use \eqref{in1'} and \eqref{in2'} to estimate the three integrals on the right-hand side of \eqref{ineq-fijn}. For $1\leq i < j\leq n$, we have
\[
\big\| f_{ij}^{(n)}(\cdot,w,z,x;t) \big\|_{\cP_0^{\otimes (n-2)}}
	\leq \dfrac{c^{n-2} t^{(n-2)(H+1)/2}}{\left[ (i-1)!(j-i-1)!(n-j)! \right]^{(H+1)/2}}f_2(w,z,x;t),
\]
and for $1\leq j<i\leq n$,
\[
 \big\| f_{ij}^{(n)}(\cdot,w,z,x;t) \big\|_{\cP_0^{\otimes (n-2)}} \leq  \dfrac{c^{n-2} t^{(n-2)(H+1)/2}}{\left[ (j-1)!(i-j-1)!(n-i)! \right]^{(H+1)/2}}f_2(z,w,x;t).
\]
Thus, introducing these two estimates in \eqref{bound-Bn}, we have
\begin{align*}
	&\bE |B_n(w,z,x;t)|^2 \leq
n!c^{n-2} t^{(n-2)(H+1)} \dfrac{1}{[(n-2)!]^{H+1}} \\
& \quad \left[ \sum_{i,j=1,i<j}^n \binom{n-2}{i-1,j-i-1}^{H+1} f_2^2(w,z,x;t)
	+ \sum_{i,j=1,i>j}^n \binom{n-2}{j-1,i-j-1}^{H+1} f_2^2(z,w,x;t)\right] \\
	& \quad \leq c^{n-2} t^{(n-2)(H+1)} \dfrac{1}{[(n-2)!]^H} \left(f_2^2(w,z,x;t) +f_2^2(z,w,x;t) \right).
\end{align*}
This proves that $\sum_{n\geq 2} B_n(w,z,x;t)$ converges in $L^2(\Omega)$. By hypercontractivity, 
\begin{align*}
	\|D^2_{z,w}u(t,x)\|_p \leq \sum_{n\geq 2}(p-1)^{\frac{n-2}{2}}\|B_n(w,z,x;t)\|_2 \leq c_1 \widetilde{f}_2(w,z,x;t) \exp \left( c_2 (p-1)^{\frac{1}{H}} t^{\frac{1+H}{H}} \right).
\end{align*}

Next, we prove \eqref{D2-bound-incr-rough}. We denote $\bar{z}=z+z'$ and $\bar{w}=w+w'$.
By \eqref{D2-series}, we have
\begin{align} \label{eq-D2-difference}
	&D^2_{w,z} u(t,x) - D^2_{\bar{w},z} u(t,x) - D_{w,\bar{z}}u(t,x) + D_{\bar{w},\bar{z}}u(t,x) \nonumber \\
	=& \sum_{n\geq 2} n(n-1) I_{n-2} \left( \widetilde{f}_n(\cdot,w,z,x;t)  - \widetilde{f}_n(\cdot,\bar{w},z,x;t) - \widetilde{f}_n(\cdot,w,\bar{z},x;t)+ \widetilde{f}_n(\cdot,\bar{w},\bar{z},x;t) \right) \nonumber \\
	=:&\sum_{n\geq 1} B^D_n(w,w',z,z',x;t).
\end{align}

Similarly to \eqref{bound-Bn}, we have
\begin{align}
\nonumber
	\bE|B^D_n(w,w',z,z',x;t)|^2 \leq n! \sum_{i,j=1, i\neq j}^{n} & \big\| f_{ij}^{(n)}(\cdot,w,z,x;t) - f_{ij}^{(n)}(\cdot,\bar{w},z,x;t)\\
\label{bound-B^Dn}
	&  - f_{ij}^{(n)}(\cdot,w,\bar{z},x;t) + f_{ij}^{(n)}(\cdot,\bar{w},\bar{z},x;t) \big\|_{\cP_0^{\otimes (n-1)}}^2.
\end{align}

Assume first that $i<j$. Using \eqref{dec-fij} we have:
\begin{align*}
	& f_{ij}^{(n)}(\pmb{x_{n-2}},w,z,x;t) - f_{ij}^{(n)}(\pmb{x_{n-2}},\bar{w},z,x;t)-	f_{ij}^{(n)}(\pmb{x_{n-2}},w,\bar{z},x;t) +f_{ij}^{(n)}(\pmb{x_{n-2}},\bar{w},\bar{z},x;t) =\\
	& \quad \int_{\{0<t_1<\ldots<t_{i-1}<\theta<t_i<\ldots< t_{j-2}<r<t_{j-1}<\ldots<t_{n-2}<t\}} d\pmb{t_{n-2}}drd\theta  \\
& \quad \quad \Big( f_{i-1} (\pmb{t_{i-1}}, \pmb{x_{i-1}},\theta,w)
g_{j-i-1}(\pmb{t_{i:j-2}},\pmb{x_{i:j-2}},\theta,w,r,z)
g_{n-j}(\pmb{t_{j-1:n-2}},\pmb{x_{j-1:n-2}},r,z,t,x) \\
& \quad -f_{i-1} (\pmb{t_{i-1}}, \pmb{x_{i-1}},\theta,\bar{w})
g_{j-i-1}(\pmb{t_{i:j-2}},\pmb{x_{i:j-2}},\theta,\bar{w},r,z)
g_{n-j}(\pmb{t_{j-1:n-2}},\pmb{x_{j-1:n-2}},r,z,t,x) \\
& \quad - f_{i-1} (\pmb{t_{i-1}}, \pmb{x_{i-1}},\theta,w)
g_{j-i-1}(\pmb{t_{i:j-2}},\pmb{x_{i:j-2}},\theta,w,r,\bar{z})
g_{n-j}(\pmb{t_{j-1:n-2}},\pmb{x_{j-1:n-2}},r,\bar{z},t,x) \\
& \quad +f_{i-1} (\pmb{t_{i-1}}, \pmb{x_{i-1}},\theta,w')
g_{j-i-1}(\pmb{t_{i:j-2}},\pmb{x_{i:j-2}},\theta,\bar{w},r,\bar{z})
g_{n-j}(\pmb{t_{j-1:n-2}},\pmb{x_{j-1:n-2}},r,\bar{z},t,x) \Big).
\end{align*}
We denote
\begin{align*}
a_1 &= f_{i-1}, (\cdot,\theta,w) \quad \quad \quad \ \
a_2= g_{j-i-1}(\cdot,\theta,w,r,z), \quad \
a_3 = g_{n-j}(\cdot,r,z,t,x), \\
b_1 &= f_{i-1} (\cdot,\theta,\bar{w}), \quad \quad \quad \ \ b_2 = g_{j-i-1}(\cdot,\theta,\bar{w},r,\bar{z}), \quad \
b_3 = g_{n-j}(\cdot,r,\bar{z},t,x), \\
	d_2 &= g_{j-i-1}(\cdot,\theta,\bar{w},r,z), \quad c_2 = g_{j-i-1}(\cdot,\theta,w,r,\bar{z}).
\end{align*}
Using the decomposition
\begin{align*}
	& a_1a_2a_3 - b_1d_2a_3 - a_1c_2b_3 + b_1b_2b_3 = (a_1-b_1) (a_2-c_2)a_3+\\
	& \quad  b_1(a_2-c_2-d_2+b_2)a_3 + (a_1-b_1)c_2(a_3-b_3)+ b_1(c_2-b_2)(a_3-b_3),
\end{align*}
and applying the Minkowski's inequality, we obtain
\begin{align*}
	& \big\| f_{ij}^{(n)}(\cdot,w,z,x;t)- f_{ij}^{(n)}(\cdot,\bar{w},z,x;t)
	-	f_{ij}^{(n)}(\cdot,w,\bar{z},x;t)+ f_{ij}^{(n)}(\cdot,\bar{w},\bar{z},x;t) \big\|_{\cP_0^{\otimes (n-2)}}  \\
& \quad \quad \quad \quad \quad \leq K_1+K_2+K_3+K_4,
\end{align*}
where
\begin{align*}
K_1&=	 \int_{0<\theta<r<t}  \Big(
\int_{T_{i-1}(\theta)} \|f_{i-1}(\pmb{t_{i-1}},\bullet,\theta,w)-f_{i-1}(\pmb{t_{i-1}},\bullet,\theta,\bar{w})\|_{\cP_0^{\otimes (i-1)}} d\pmb{t_{i-1}}\Big)\\
& \
\Big(\int_{\{\theta<t_i<\ldots< t_{j-2}<r\}} \|g_{j-i-1}(\pmb{t_{i:j-2}},\bullet,\theta,w,r,z)-
g_{j-i-1}(\pmb{t_{i:j-2}},\bullet,\theta,w,r,\bar{z})\|_{\cP_0^{\otimes (j-i-1)}} d\pmb{t_{i:j-2}}\Big)\\
& \
\Big(\int_{\{r<t_{j-1}<\ldots<t_{n-2}<t\}} \| g_{n-j}(\pmb{t_{j-1:n-2}},\bullet,r,z,t,x) \|_{\cP_0^{\otimes (n-j)}} d\pmb{t_{j-1:n-2}} \Big) d\theta dr, \\
K_2&=	\int_{0<\theta<r<t} \Big(\int_{T_{i-1}(\theta)} \|f_{i-1}(\pmb{t_{i-1}},\bullet,\theta,\bar{w})\|_{\cP_0^{\otimes (i-1)}} d\pmb{t_{i-1}} \Big)  \\
& \
\Big(\int_{\{\theta<t_i<\ldots< t_{j-2}<r\}} \|g_{j-i-1}(\pmb{t_{i:j-2}},\bullet,\theta,w,r,z) - g_{j-i-1}(\pmb{t_{i:j-2}},\bullet,\theta,w,r,\bar{z})-\\
& \quad
g_{j-i-1}(\pmb{t_{i:j-2}},\bullet,\theta,\bar{w},r,z)+
g_{j-i-1}(\pmb{t_{i:j-2}},\bullet,\theta,\bar{w},r,\bar{z})\|_{\cP_0^{\otimes (j-i-1)}} d\pmb{t_{i:j-2}}\Big)\\
& \ \Big(\int_{\{r<t_{j-1}<\ldots<t_{n-2}<t\}} \|g_{n-j}(\pmb{t_{j-1:n-2}},\bullet,r,z,t,x)\|_{\cP_0^{\otimes (n-j)}} d\pmb{t_{j-1:n-2}} 	\Big)d\theta dr, \\
K_3&=\int_{0<\theta<r<t}
\Big(\int_{T_{i-1}(\theta)} \|f_{i-1}(\pmb{t_{i-1}},\bullet,\theta,w)-f_{i-1}(\pmb{t_{i-1}},\bullet,\theta,\bar{w})\|_{\cP_0^{\otimes (i-1)}} d\pmb{t_{i-1}}\Big)\\
& \
\Big(\int_{\{\theta<t_i<\ldots< t_{j-2}<r\}} \|g_{j-i-1}(\pmb{t_{i:j-2}},\bullet,\theta,w,r,\bar{z})\|_{\cP_0^{\otimes (j-i-1)}} d\pmb{t_{i:j-2}}\Big)\\
& \
\Big(\int_{\{r<t_{j-1}<\ldots<t_{n-2}<t\}} \| g_{n-j}(\pmb{t_{j-1:n-2}},\bullet,r,z,t,x)-g_{n-j} (\pmb{t_{j-1:n-2}},\bullet,r,\bar{z},t,x) \|_{\cP_0^{\otimes (n-j)}} d\pmb{t_{j-1:n-2}} \Big) d\theta dr, \\
K_4&=\int_{0<\theta<r<t} \Big(\int_{T_{i-1}(\theta)} \|f_{i-1}(\pmb{t_{i-1}},\bullet,\theta,\bar{w})\|_{\cP_0^{\otimes (i-1)}} d\pmb{t_{i-1}}\Big) \\
&\ \Big( \int_{\{r<t_{j-1}<\ldots<t_{n-2}<t\}} \|g_{j-i-1}(\pmb{t_{j-1:n-2}},\bullet,\theta,w,r,\bar{z})-
g_{j-i-1}(\pmb{t_{j-1:n-2}},\bullet,\theta,\bar{w},r,\bar{z})\|_{\cP_0^{\otimes (n-j)}} d\pmb{t_{j-1:n-2}}\Big) \\
& \
\Big(\int_{\{\theta<t_i<\ldots< t_{j-2}<r\}} \|g_{n-j}(\pmb{t_{i:j-2}},\bullet,r,z,t,x) - g_{n-j}(\pmb{t_{i:j-2}},\bullet,r,\bar{z},t,x) \|_{\cP_0^{\otimes (j-i-1)}} d\pmb{t_{i:j-2}}\Big)
d\theta dr.
\end{align*}

We now use inequalities \eqref{in1'}, \eqref{in2'},
\eqref{ineq-norm-difference f_j},
\eqref{ineq-norm-difference g_j},
\eqref{ineq-norm-difference g_j-2} and
\eqref{rectangular-diff}. Thus, for $i<j$,
\begin{align*}
	& \big\| f_{ij}^{(n)}(\cdot,w,z,x;t)- f_{ij}^{(n)}(\cdot,\bar{w},z,x;t)
	-	f_{ij}^{(n)}(\cdot,w,\bar{z},x;t)+ f_{ij}^{(n)}(\cdot,\bar{w},\bar{z},x;t) \big\|_{\cP_0^{\otimes (n-2)}}  \leq \\
& \quad \quad \quad \frac{c^{n-2}t^{(n-2)\frac{H+1}{2}}}{[(i-1)! (j-i-1)!(n-j)!]^{\frac{H+1}{2}}}
\int_{0<\theta<r<t}H(\theta,w,w',r,z,z',x;t)d\theta dr.
\end{align*}
A similar inequality holds for $j<i$, by replacing $(i,\theta,w,w', j,r,z,z')$ and
$(j,r,z,z',i,\theta,w,w')$. Returning to \eqref{bound-B^Dn}, we conclude that
\begin{align*}
& \bE|B_n^D(w,w',z,z',x;t)|^2 \leq \frac{C^{n-2}t^{(n-2)(H+1)}}{[(n-2)!]^{H+1}}\\
& \quad
\left\{\left( \int_{0<\theta<r<t}H(\theta,w,w',r,z,z',x;t)d\theta dr \right)^2+
\left( \int_{0<r<\theta<t}H(r,z,z',\theta,w,w'x;t)d\theta dr \right)^2\right\}.
\end{align*}
Finally, \eqref{D2-bound-incr-rough} follows by hypercontractivity.
\end{proof}

Next, we introduce two notations from \cite{NXZ}. For $t \in \bR_{+}$ and $x' \in \bR$, let $\Phi_{t,x'}$ be an operator on $\cB(\bR)$, the set of real Borel measurable functions on $\bR$, given by
\begin{align*}
	(\Phi_{t,x'} g) (x)
	=& g(x+x') {\bf 1}_{\{|x'| > \sqrt{t}\}} + g(x) N_t(x'),
\end{align*}
for $g \in \cB(\bR)$. For $0 < r < s$ and $y',z' \in \bR$, let $\Lambda_{r,z',s,y'}: \cB(\bR) \times \cB(\bR) \to \cB(\bR^2)$ be an operator given by
\begin{align*}
	\Lambda_{r,z',s,y'}(g_1,g_2) (x,y)
	=& g_1(x) \left( \Phi_{s-r,y'} g_2 \right)(y) N_r(z')
	+ g_1(x) \left( \Phi_{s-r,y'} \Phi_{s-r,-z'} g_2 \right)(y) \\
	& + \left( \Phi_{t-s,-y'} g_1 \right)(x) g_2(y+y') N_r(z')
	+ \left( \Phi_{t-s,-y'} g_1 \right)(x) \left( \Phi_{s-r,-z'} g_2 \right)(y+y'),
\end{align*}
for $g_1,g_2 \in \cB(\bR)$.
The following result is the analogue of Proposition 4.1 of \cite{NXZ} for the time-independent noise.

\begin{proposition}
\label{prop-4-1}
If Assumption C holds, then for any $p\geq 2$, $t>0$ and $x,z,z',w,w' \in \bR$,
\[
\|D_{z+z'}u(t,x)-D_z u(t,x) \|_p \leq C_1(t) \int_0^t (\Phi_{t-r,z'}p_{4(t-r)})(z-x)dr
\]
and
\begin{align*}
& \| D^2_{w+w',z+z'} u(t,x) - D^2_{w+w',z} u(t,x) - D_{w,z+z'}u(t,x)+ D_{w,z}u(t,x)  \big\|_p \leq \\
& \quad C_1(t) \left( \int_{0<\theta<r<t}\Lambda_{\theta,w',r,z'}(p_{4(t-r)},p_{4(r-\theta)})(x-z,z-w)d\theta dr+\right.\\
& \quad \quad \quad \left. \int_{0<r<\theta<t}\Lambda_{r,z',\theta,w'}(p_{4(t-\theta)},p_{4(\theta-r)})(x-w,w-z)drd\theta \right),
\end{align*}
where $C_1(t)$ is the same as in Lemma \ref{Thm-existence-heat-rough}.
\end{proposition}

\begin{proof}
We use \eqref{D-bound-incr-rough} and \eqref{D2-bound-incr-rough}. Since the integrands appearing in these two relations are the same as in (4.4) and (4.5) of \cite{NXZ}, we can bound these integrands exactly as in the proof of Proposition 4.1 of \cite{NXZ}.
\end{proof}

\subsection{Proof of Theorem \ref{main} under Assumption C}
\label{section-main-C}

In this section, we give the proof of Theorem \ref{main} in the rough noise case.

\medskip

(i) (Limiting Covariance) We use a similar method as in the proof of Proposition 1.2 of \cite{NSZ}, based on the Feynman-Kac formula given in Appendix \ref{appA}. For simplicity, we assume that $t=s$. The general case is similar.

Recall that we have the chaos expansion $F_R(t) = \sum_{n\geq 1} {\bf J}_{n,R} (t)$, where ${\bf J}_{n,R}(t)$ is given by \eqref{def-JnR}. We examine the first chaos. By direct calculation,
\begin{align} \label{eq-norm-J_1R(t)}
	\bE \left[ \left( {\bf J}_{1,R} (t) \right)^2 \right]
	=& \langle g_{1,R}(\cdot;t),g_{1,R}(\cdot;t) \rangle_{\cP_0} \nonumber \\
	=& \int_{B_R^2} \int_0^t  \int_0^t \int_{\bR}  e^{-i\xi (x-x')} \cF p_{t-t_1}(\xi) \overline{\cF p_{t-t_1'}(\xi)} \mu(d\xi) dt_1 dt_1' dxdx' \nonumber \\
	=& \int_0^t  \int_0^t \int_{\bR}  e^{-\frac{t_1+t_1'}{2}|\xi|^2} \left|\int_{B_R}e^{-i\xi x}dx\right|^2 \mu(d\xi)dt_1 dt_1' \nonumber \\
	=& 4\pi R \int_0^t  \int_0^t  \int_{\bR} e^{-\frac{t_1+t_1'}{2}|\xi|^2} \ell_R(\xi) \mu(d\xi)dt_1dt_1'.
\end{align}
Using \eqref{lem3-2-NSZ},
\begin{align} \label{ineq-norm-J_1R(t)}
	\bE \left[ \left( {\bf J}_{1,R} (t) \right)^2 \right]
	\le 4\pi R \int_0^t \int_0^t \int_{\bR} \ell_R(\xi) \mu(d\xi)dt_1 dt_1'
	\le 4\pi R t^2 \left( \varepsilon + \dfrac{C(\varepsilon)}{R} \right),
\end{align}
which implies that
\[
\lim_{R \to \infty} \frac{1}{R} \bE \left[ {\bf J}_{1,R} (t)^2 \right] = 0.
\]
For the variance of ${\bf J}_{n,R} (t)$ with $n \ge 2$, we have
\begin{align*}
	& \bE \left[ \left( {\bf J}_{n,R} (t) \right)^2 \right]
	=n! \langle  g_n(\cdot;t), \widetilde{g}_n(\cdot;t) \rangle_{\cP_0^{\otimes n}}\\
	=& n! \int_{B_R^2} dxdx' \int_{\bR^n} \cF f_n(\cdot,x;t) (\xi_1, \ldots, \xi_n) \overline{\cF \widetilde{f}_n(\cdot,x';t)(\xi_1, \ldots, \xi_n)} \mu(d\xi_1) \ldots \mu(d\xi_n)  \\
	=& \sum_{\rho \in \Sigma_n} \int_{B_R^2} dxdx' \int_{\bR^n} \mu(d\xi_1) \ldots \mu(d\xi_n) \int_{T_n(t)} d\pmb{t_n} \int_{T_n(t)} d\pmb{t_n'} \\
	& \times e^{-i(\xi_1+\ldots+\xi_n)\cdot (x-x')} \prod_{j=1}^n \cF p_{t_{j+1}-t_j}(\xi_1+\ldots+\xi_j) \overline{\cF p_{t_{j+1}'-t_j'}(\xi_{\rho(1)}+\ldots+\xi_{\rho(j)})},
\end{align*}
with
the convention $t_{n+1} = t_{n+1}' = t$.
 From this formula, we can see that the integrand only depends on $x-x'$. We change of variable $x' = x-z$, we have the following formula:
\begin{align*}
	\int_{B_R^2} h(x-x') dxdx'
	= \int_{\bR} \int_{B_R \cap (B_R+z)} h(z) dxdz
	= \int_{\bR} h(z) \text{Leb}([-R,R] \cap [-R+z,R+z]) dz
\end{align*}
Hence, using the formula, we have
\begin{align*}
	&\bE \left[ \left( {\bf J}_{n,R} (t) \right)^2 \right]
	= \sum_{\rho \in \Sigma_n} \int_{\bR} \text{Leb}([-R,R] \cap [-R+z,R+z]) dz \int_{\bR^n} \mu(d\xi_1) \ldots \mu(d\xi_n) \\
	& \quad \int_{T_n(t)}\int_{T_n(t)}
	 e^{-i(\xi_1+\ldots+\xi_n)\cdot z} \prod_{j=1}^n \cF p_{t_{j+1}-t_j}(\xi_1+\ldots+\xi_j) \overline{\cF p_{t_{j+1}'-t_j'}(\xi_{\rho(1)}+\ldots+\xi_{\rho(j)})} d\pmb{t_n}  d\pmb{t_n'}.
\end{align*}
Note that
$\prod_{j=1}^n \cF p_{t_{j+1}-t_j}(\xi_1+\ldots+\xi_j)
	= \bE \left[ e^{i\sum_{j=1}^n B_{t-t_j} \xi_j} \right]$,
where $(B_t)_{t\geq 0}$ is a Brownian motion.
Hence,
\begin{align} \label{eq-second moment of JnR}
	\bE \left[ \left( {\bf J}_{n,R} (t) \right)^2 \right]
	=& \sum_{\rho \in \Sigma_n} \int_{\bR} \text{Leb}([-R,R] \cap [-R+z,R+z]) dz \int_{\bR^n} \mu(d\xi_1) \ldots \mu(d\xi_n) \nonumber \\
	& \int_{T_n(t)} d\pmb{t_n} \int_{T_n(t)} d\pmb{t_n'}
	e^{-i(\xi_1+\ldots+\xi_n)\cdot z}
	\bE \left[ e^{i\sum_{j=1}^n B_{t-t_j} \xi_j} \right]
	\bE \left[ e^{i\sum_{j=1}^n B_{t-t_j'} \xi_{\rho(j)}} \right] \nonumber \\
=&\dfrac{1}{n!} \int_{\bR} \text{Leb}([-R,R] \cap [-R+z,R+z])  \bE \left[ \left( \cI_{t,t}^{1,2}(z) \right)^n \right]dz.
\end{align}
Similarly to (3.19) of \cite{NSZ}, it can be proved that
\begin{align} \label{ineq-series-claim}
	\sum_{n\geq 2} \dfrac{1}{n!} \int_{\bR} \bE \left[ \left| \cI_{t,t}^{1,2}(z) \right|^n \right] dz < \infty.
\end{align}
Since $$\dfrac{\text{Leb}([-R,R] \cap [-R+z,R+z])}{R}
	= \dfrac{\min\{0, 2R-|z|\}}{R}$$
is bounded by 2 and converges to 2 as $R \to \infty$, by the dominated convergence theorem,
\begin{align*}
	\lim_{R \to \infty} \dfrac{1}{R}\sum_{n\geq 2} \bE \left[ \left( {\bf J}_{n,R} (t) \right)^2 \right]
	= 2\sum_{n\geq 2}\dfrac{1}{n!} \int_{\bR}  \bE \left[ \left( \cI_{t,t}^{1,2}(z) \right)^n \right]dz.
\end{align*}

\medskip

(ii) (QCLT) By applying a version of Proposition 2.4 of \cite{NXZ} for the time-independent noise, we have:
\[
d_{TV}\left(\frac{F_R(t)}{\sigma_R(t)},Z\right)\leq \frac{2}{\sigma_R^2(t)} \sqrt{{\rm Var}(\langle DF_R(t),-DL^{-1}F_R(t)\rangle_{\cH})} 
\]
and
\[
{\rm Var}(\langle DF_R(t),-DL^{-1}F_R(t)\rangle_{\cH})\leq 3\cA
\]
where $Z \sim N(0,1)$ and
\begin{align*}
\cA&=C_H^3\int_{\bR^6}
\| D_{z,y}^2 F_R(t)-D_{z,y'}^2 F_R(t)- D_{z',y}^2 F_R(t)+D_{z',y'}^2 F_R(t)\|_4 \\
& \quad \quad \quad \| D_{w,y}^2 F_R(t)-D_{w,y'}^2 F_R(t)- D_{w',y}^2 F_R(t)+D_{w',y'}^2 F_R(t)\|_4 \\
& \quad \quad \quad \|D_z F_R(t)-D_{z'} F_R(t)\|_4 \|D_{w} F_R(t)-D_{w'} F_R(t)\|_4\\
& \quad \quad \quad
|y-y'|^{2H-2} |z-z'|^{2H-2} |w-w'|^{2H-2}dydy' dzdz' dw dw'.
\end{align*}

Since $\sigma_R^2(t) \sim K(t,t)R$ (by part (i)), it is enough to prove that
\begin{equation}
\label{A-bound}
\cA \leq C_t R,
\end{equation}
where $C_t>0$ is a constant depending on $t$ (and $H$).

Note that $D_{z}F_{R}(t)-D_{z'}F_{R}(t)=\int_{B_R}\big(D_z u(t,x)-D_{z'} u(t,x)\big) dx$. Therefore, by Minkowski's inequality,
\[
\|D_{z}F_{R}(t)-D_{z'}F_{R}(t)\|_4 \leq \int_{B_R}\|D_z u(t,x)-D_{z'}u(t,x)\|_4 dx.
\]
Similarly,
\begin{align*}
& \| D_{z,y}^2 F_R(t)-D_{z,y'}^2 F_R(t)- D_{z',y}^2 F_R(t)+D_{z',y'}^2 F_R(t)\|_4 \leq \\
& \quad \int_{B_R}\|D_{z,y}^2 u(t,x)-D_{z,y'}^2 u(t,x)- D_{z',y}^2 u(t,x)+D_{z',y'}^2 u(t,x) \big) \|_4 dx.
\end{align*}
Hence $\cA \leq \cA'$, where, after a change of variables,
\begin{align*}
& \cA' = C_H^3 \int_{B_R^4}\int_{\bR^6}\|D_{z+z'} u(t,x_1)-D_{z} u(t,x_1)\|_4 \|D_{w+w'} u(t,x_2)-D_{w} u(t,x_2)\|_4\\
& \quad \| D_{z+z',y+y'}^2 u(t,x_3)-D_{z,y+y'}^2 u(t,x_3)- D_{z+z',y}^2 u(t,x_3)+D_{z,y}^2 u(t,x_3)\|_4 \\
& \quad \| D_{w+w',y+y'}^2 u(t,x_4)-D_{w,y+y'}^2 u(t,x_4)- D_{w+w',y}^2 u(t,x_4)+D_{w,y}^2 u(t,x_4)\|_4 \\
& \quad \quad \quad
|y'|^{2H-2} |z'|^{2H-2} |w'|^{2H-2}
dydy' dzdz' dw dw' dx_1 dx_2 dx_3 dx_4.
\end{align*}
Therefore, it is enough to prove that
\[
\cA' \leq C_t R.
\]

By Proposition \ref{prop-4-1}, we have:
\begin{align*}
\cA' & \leq C_1^4(t)\int_{[-R,R]^4} \int_{\bR^6}\left(\int_0^t (\Phi_{t-r',z'}p_{4(t-r')})(z-x_1)dr'\right)\left( \int_0^t
(\Phi_{t-\theta',w'}p_{4(t-\theta')})(w-x_2)d\theta'\right)\\
& \quad \times \left( \int_{0<r<s<t} \Lambda_{r,z',s,y'}(p_{4(t-s)},p_{4(s-r)})(x_3-y,y-z)drds+\right.\\
& \qquad \qquad \qquad \left.\int_{0<s<r<t} \Lambda_{s,y',r,z'}(p_{4(t-r)},p_{4(r-s)})(x_3-z,z-y)dsdr \right)\\
& \quad \times \left( \int_{0<\theta<s'<t} \Lambda_{\theta,w',s',y'}(p_{4(t-s')},p_{4(s'-\theta)})(x_4-y,y-w)d\theta ds'+\right.\\
& \qquad \qquad \qquad  \left.\int_{0<s'<\theta<t} \Lambda_{s',y',\theta,w'}(p_{4(t-\theta)},p_{4(\theta-s')})(x_4-w,w-y)ds' d\theta \right)\\
& \quad  |y'|^{2H-2}|z'|^{2H-2} |w'|^{2H-2}dydy' dzdz' dwdw' dx_1 dx_2 dx_3 dx_4=:C_1^4(t) \sum_{i=1}^4 \cA_{i}.
\end{align*}
The 4 terms $\cA_{i},i=1,2,3,4$ correspond to the integrals over the sets $\{r<s,\theta<s'\}$, $\{r<s,s'<\theta\}$,
$\{s<r,\theta<s'\}$, $\{s<r,s'<\theta\}$, respectively.

We discuss only $\cA_{1}$, the other 3 terms being similar. For any $r,r',s,s',\theta,\theta'\in [0,t]$ fixed with $r<s$ and $\theta<s'$, the integral on $[-R,R]^4 \times \bR^6$ coincides with the one given as upper bound for $\cA_0$ on the display equation on page 33, lines 9-11 of \cite{NXZ}. This integral is shown in this reference to be bounded by $C_t R \cB_0$, where $C_t>0$ is a constant depending on $t$ and $H$, and
$\cB_0$ is given on the display after (4.8) of \cite{NXZ}:
\begin{align*}
\cB_0&=(t-r')^a (t-\theta')^a \left[ (s-r)^a r^a +(s-r)^{2a}+(t-s)^ar^a +(t-s)^a(s-r)^a\right]\\
&\qquad \qquad \qquad \quad \quad \left[ (s'-\theta)^a \theta^a +(s'-\theta)^{2a}+(t-s')^a \theta^a +(t-s')^a(s'-\theta)^a\right],
\end{align*}
where $a=\frac{H}{2}-\frac{1}{4}$. Hence,
\begin{align*}
\cA_1 & \leq CR \int_{[0,t]^6}1_{\{r<s,\theta<s'\}}\cB_0 drdr' dsds' d\theta d\theta'\\
&=CR \int_{[0,t]^6} \varphi(r,s,\theta)\varphi(r',s',\theta')drdr' dsds' d\theta d\theta'=CR \left(\int_{[0,t]^3}\varphi(r,s,\theta)drdsd\theta\right)^2,
\end{align*}
where for the second line we used the change of variables $\theta \leftrightarrow r'$, and we denoted
\[
\varphi(r,s,\theta)=1_{\{r<s\}}(t-\theta)^a \left[ (s-r)^a r^a +(s-r)^{2a}+(t-s)^ar^a +(t-s)^a(s-r)^a\right].
\]
Finally, it is not difficult to see that
\[
\int_{[0,t]^3}\varphi(r,s,\theta)drdsd\theta=C t^{3a+3}.
\]
This shows that $\cA_1 \leq C_t R$ and concludes the proof of the theorem.

\medskip

(iii) Use the same method as Theorem 1.1 of \cite{NSZ}. We first show the tightness, and then establish the finite dimensional convergence.

We will use frequently the fact that
\begin{equation}
\label{ell-R}
\left|\int_{B_R} e^{-i \xi x} dx\right|^2 =4\pi R \ell_{R}(\xi) \quad \mbox{where} \quad
\ell_R(\xi) = \dfrac{\sin^2(R\xi)}{\pi R \xi^2}.
\end{equation}

We will use the following result.

\begin{lemma}[Lemma 3.2 of \cite{NSZ}]
\label{lem3-2-NSZ}
For any $\varepsilon > 0$, there exists a constant $C(\varepsilon)>0$ depending on $\varepsilon$ such that $$\int_{\bR} \ell_R(\xi) \mu(d\xi)\leq \e+\frac{C(\e)}{R}.$$
\end{lemma}

{\em Step 1 (tightness).} We prove that for any $R>T$, $p\geq 2$ and $0<s<t<T$,
\begin{equation} \label{rough-tight}
	\|F_R(t)-F_R(s)\|_p \leq C R^{1/2} (t-s)^{1/2},
\end{equation}
where $C>0$ is a constant depending on $T,H$ and $p$.
By Kolmogorov's continuity theorem, the process $F_R=\{F_R(t)\}_{t\in [0,T]}$ has a continuous modification (which we denote also $F_R$), for any $R>0$. The family $\{F_R\}_{R>0}$ is tight in $C[0,T]$, by Theorem 12.3 of \cite{billingsley68}.

Using the chaos expansion, we can still write $F_R(t)-F_R(s)= \sum_{n\geq 1} I_n (g_{n,R}(\cdot;t,s))$, where $g_{n,R} = g_{n,R}^{(1)} + g_{n,R}^{(2)}$ are given in \eqref{eq-decomposition-1}. Hence, by orthogonality and hypercontractivity,
\begin{align}
\label{rough-tight-1}
	\|F_R(t) - F_R(s)\|_p
	\le \sum_{n\geq 1} (p-1)^{n/2}\|I_n (g_{n,R}^{(1)}(\cdot;t,s))\|_2 + \sum_{n\geq 1}  (p-1)^{n/2} \|I_n (g_{n,R}^{(2)}(\cdot;t,s))\|_2.
\end{align}

We compute the first sum in \eqref{rough-tight-1}. For $n=1$, we have:
\begin{align*}
	g_{1,R}^{(1)} (x_1;t,s) = \int_{B_R} \int_0^s (p_{t-t_1}(x-x_1) - p_{s-t_1}(x-x_1)) dt_1 dx.
\end{align*}
Recalling \eqref{ell-R}, we have
\begin{align*}
	& \|I_1 (g_{1,R}^{(1)}(\cdot;t,s))\|_2^2
	= \|g_{1,R}^{(1)}(\cdot;t,s)\|_{\cP_0}^2 \\
	=& \int_{B_R^2} dxdx' \int_0^s dt_1 \int_0^s dt_1' \int_{\bR} \mu(d\xi) e^{-i(x-x')\xi} \left( \cF p_{t-t_1}(\xi) - \cF p_{s-t_1}(\xi) \right) \left( \cF p_{t-t_1'}(\xi) - \cF p_{s-t_1'}(\xi) \right) \\
	=& 4\pi R \int_0^s dt_1 \int_0^s dt_1' \int_{\bR} \ell_R(\xi) \mu(d\xi) \left( \cF p_{s-t_1}(\xi) - \cF p_{t-t_1}(\xi) \right) \left( \cF p_{s-t_1'}(\xi) - \cF p_{t-t_1'}(\xi) \right).
\end{align*}
Note that, for any $t_1 \in [0,1]$,
\begin{align*}
	\cF p_{s-t_1}(\xi) - \cF p_{t-t_1}(\xi)
	=& e^{ -\frac{(s-t_1)|\xi|^2}{2} } \left(1-e^{-\frac{(t-s)|\xi|^2}{2}}  \right)
	\le \frac{(t-s)|\xi|^2}{2} e^{ -\frac{(s-t_1)|\xi|^2}{2} },
\end{align*}
and $\cF p_{s-t_1}(\xi) - \cF p_{t-t_1}(\xi) \in [0,1]$. Hence,
\begin{align*}
	\|I_1 (g_{1,R}^{(1)}(\cdot;t,s))\|_2^2
	\le& 2\pi R (t-s) s \int_{\bR} \ell_R(\xi) |\xi|^2 \left(\int_0^s e^{ -\frac{(s-t_1)|\xi|^2}{2} }  dt_1 \right) \mu(d\xi)
	\leq 4\pi R(t-s)s  \int_{\bR} \ell_R(\xi) \mu(d\xi).
\end{align*}
Using Lemma \ref{lem3-2-NSZ}, we obtain:
\begin{align} \label{ineq-I_1-g1-n=1-rough-tight}
	\|I_1 (g_{1,R}^{(1)}(\cdot;t,s))\|_2^2 \le C R (t-s),
\end{align}
where $C$ is a constant that depends on $T,H$.
For $n \ge 2$, we have
\begin{align*}
	& \|I_n (g_{n,R}^{(1)}(\cdot;t,s))\|_2^2
	= n! \langle g_{n,R}^{(1)}(\cdot;t,s), \tilde g_{n,R}^{(1)}(\cdot;t,s) \rangle_{\cP_0} \\
	=& \sum_{\rho \in \Sigma_n} \int_{B_R^2} dxdx' \int_{T_n(s)} d\pmb{t} \int_{T_n(s)} d\pmb{t'} \int_{\bR^n} \mu(d\xi_1) \ldots \mu(d\xi_n) e^{-i(x-x') (\xi_1+\ldots+\xi_n)} \\
	& \prod_{j=1}^{n-1} \cF p_{t_{j+1}-t_j}(\xi_1+\ldots+\xi_j) \left( \cF p_{t-t_n}(\xi_1+ \ldots + \xi_n) - \cF p_{s-t_n}(\xi_1+ \ldots + \xi_n) \right) \\
	& \times \prod_{j=1}^{n-1} \cF p_{t_{j+1}'-t_j'}(\xi_{\rho(1)}+\ldots+\xi_{\rho(j)}) \left( \cF p_{t-t_n'}(\xi_1+ \ldots + \xi_n) - \cF p_{s-t_n'}(\xi_1+ \ldots + \xi_n) \right).
\end{align*}
Note that
$\prod_{j=1}^n \cF p_{t_{j+1}-t_j}(\xi_1+\ldots+\xi_j) = \bE \left[ e^{i\sum_{j=1}^n (B_t-B_{t_j}) \xi_j} \right]$, where $(B_t)_{t\geq 0}$ is a Brownian motion. We have
\begin{align*}
	\|I_n (g_{n,R}^{(1)}(\cdot;t,s))\|_2^2
	=& \sum_{\rho \in \Sigma_n} \int_{B_R^2} dxdx' \int_{T_n(s)} d\pmb{t} \int_{T_n(s)} d\pmb{t'} \int_{\bR^n} \mu(d\xi_1) \ldots \mu(d\xi_n) e^{-i(x-x') (\xi_1+\ldots+\xi_n)} \\
	& \times \left( \bE \left[ e^{i\sum_{j=1}^n (B_s-B_{t_j}) \xi_j} \right] - \bE \left[ e^{i\sum_{j=1}^n (B_t-B_{t_j}) \xi_j} \right] \right) \\
	&\times \left( \bE \left[ e^{i\sum_{j=1}^n (B_s-B_{t_j'}) \xi_{\rho(j)}} \right] - \bE \left[ e^{i\sum_{j=1}^n (B_t-B_{t_j'}) \xi_{\rho(j)}} \right] \right).
\end{align*}
Since
\begin{align*}
	\bE \left[ e^{i\sum_{j=1}^n (B_s-B_{t_j}) \xi_j} \right] - \bE \left[ e^{i\sum_{j=1}^n (B_t-B_{t_j}) \xi_j} \right] \in [0,1],
\end{align*}
we can bound the second parenthesis above by 1, so that the remaining $d\pmb{t'}$ integral is equal to $s^n/n!$. Together with the sum over $\rho$, this yields the factor $s^n$. For the first parenthesis, we use relation (5.5) of \cite{NSZ}:
\begin{align*}
	\bE \left[ e^{i\sum_{j=1}^n (B_s-B_{t_j}) \xi_j} \right] - \bE \left[ e^{i\sum_{j=1}^n (B_t-B_{t_j}) \xi_j} \right] \leq \frac{t-s}{2}(\sum_{j=1}^n \xi_j)^2 \bE \left[ e^{i\sum_{j=1}^n (B_s-B_{t_j}) \xi_j} \right].
\end{align*}
Since the new integrand of the $d\pmb{t}$ integral is symmetric in $(t_1,\ldots,t_n)$, we obtain:
\begin{align*}
	\|I_n (g_{n,R}^{(1)}(\cdot;t,s))\|_2^2
	& \leq \dfrac{2(t-s)s^n \pi R}{n!} \int_{[0,s]^n} d\pmb{t} \int_{\bR^n} \ell_R(\sum_{j=1}^n\xi_j) \mu(d\xi_1) \ldots \mu(d\xi_n) \\
	& \quad \quad \quad \times (\sum_{j=1}^n \xi_j)^2 \bE \left[ e^{i\sum_{j=1}^n (B_s-B_{t_j}) \xi_j} \right] \\
	& \leq  \dfrac{2(t-s)s^n}{n!} \frac{R}{T} \int_{[0,s]^n} d\pmb{t} \int_{\bR^n} \mu(d\xi_1) \ldots \mu(d\xi_n) \bE \left[ e^{i\sum_{j=1}^n (B_s-B_{t_j}) \xi_j} \right],
\end{align*}
where we used the inequality $\ell_R(\xi) \le \frac{1}{\pi R |\xi|^2}\leq \frac{1}{\pi T |\xi|^2}$ for any $R>T$.
Note that
\begin{align*}
	\bE \left[ e^{i\sum_{j=1}^n (B_s-B_{t_j}) \xi_j} \right]
	= \bE \left[ e^{-\mathrm{Var} \sum_{j=1}^n (B_s-B_{t_j}) \xi_j} \right].
\end{align*}
Using relation (3.9) of \cite{NSZ} (with $H_0 = 1/2$), we obtain:
\begin{align} \label{ineq-I_1-g1-n>1-rough-tight}
	\|I_n (g_{n,R}^{(1)}(\cdot;t,s))\|_2^2
	\le 2(t-s) s^n \frac{R}{T} C^n \dfrac{s^{nH}}{\Gamma(nH+1)},
\end{align}
where $C$ is a constant that depends on $H$.
Therefore, by \eqref{ineq-I_1-g1-n=1-rough-tight} and \eqref{ineq-I_1-g1-n>1-rough-tight},  for any $R>T$,
\begin{align}
\label{ineq-I_n-norm-rough-1}
	\sum_{n\geq 1} (p-1)^{n/2}\|I_n (g_{n,R}^{(1)}(\cdot;t,s))\|_2
	\le C R^{1/2} (t-s)^{1/2},
\end{align}
where $C$ is a constant that depends on $T$ and $H$.

Next, we estimate the second sum in \eqref{rough-tight-1}. For $n=1$, noting that
\begin{align*}
	g_{1,R}^{(2)} (x_1;t,s) = \int_{B_R} \int_s^t p_{t-t_1}(x-x_1) dt_1 dx,
\end{align*}
we have
\begin{align*}
	\|I_1 (g_{1,R}^{(2)}(\cdot;t,s))\|_2^2&= \int_{B_R^2} \int_s^t \int_s^t e^{-i\xi(x-x')} \cF p_{t-t_1}(\xi) \cF p_{t-t_1'}(\xi)\mu(d\xi) dt_1 dt_1' dxdx'\\
&=4\pi R \int_s^t \int_s^t \int_{\bR} e^{-\frac{t-t_1}{2}|\xi|^2} e^{-\frac{t-t_1'}{2}|\xi|^2}\ell_{R}(\xi)\mu(d\xi) dt_1 dt_1'\\
&\leq 4\pi R (t-s)^2 \int_{\bR}\ell_R(\xi)\mu(d\xi).
\end{align*}
By Lemma \ref{lem3-2-NSZ}, $\int_{\bR}\ell_R(\xi)\mu(d\xi)\leq \e+\frac{C(\e)}{T}$ for any $R>T$. Hence, for any $R>T$,
\begin{align}
\label{ineq-I_1-g2-n=1-rough-tight}
	\|I_1 (g_{1,R}^{(2)}(\cdot;t,s))\|_2^2
	\le C R (t-s)^2.
\end{align}

For $n \ge 2$, we use the identity $\prod_{j=1}^n \cF p_{t_{j+1}-t_j}(\xi_1+\ldots+\xi_j) = \bE \left[ e^{i\sum_{j=1}^n B_{t-t_j} \xi_j} \right]$.
By direct calculation,
\begin{align*}
	& \|I_n (g_{n,R}^{(2)}(\cdot;t,s))\|_2^2
	= n! \langle g_{n,R}^{(2)}(\cdot;t,s), \tilde g_{n,R}^{(2)}(\cdot;t,s) \rangle_{\cP_0^{\otimes n}} \\
	& \quad =4 \pi R \sum_{\rho \in S_n} \int_{T_n(t)} d\pmb{t} \int_{T_n(t)} d\pmb{t'} \int_{(\bR^{d})^n} \mu(d\xi_1) \ldots \mu(d\xi_n)
	\bE \left[ e^{i\sum_{j=1}^n B_{t-t_j} \xi_j} \right]
	\bE \left[ e^{i\sum_{j=1}^n B_{t-t_j'} \xi_{\rho(j)}} \right]
	\nonumber \\
	& \quad \quad \quad \times 1_{[s,t]} (t_n) 1_{[s,t]}(t_n') \ell_R (\sum_{j=1}^n \xi_j) \\
&  \quad \leq  4 \pi R n! \int_{T_n(t)} d\pmb{t} \int_{(\bR^{d})^n} \mu(d\xi_1) \ldots \mu(d\xi_n)
	\bE \left[ e^{i\sum_{j=1}^n B_{t-t_j} \xi_j} \right]
	\ell_R (\sum_{j=1}^n \xi_n) \int_{T_n(t)} 1_{[s,t]}(t_n') d\pmb{t'}.
\end{align*}
Note that for $0 \le s < t \le T$, we have
\begin{align*}
	\int_{T_n(t)} 1_{[s,t]}(t_n') d\pmb{t'}
	= \int_s^t \dfrac{x^{n-1}}{(n-1)!} dx
	= \dfrac{t^n - s^n}{n!}
	\le \dfrac{(t-s) n t^{n-1}}{n!}.
\end{align*}
Hence, we have
\begin{align*}
	& \|I_n (g_{n,R}^{(2)}(\cdot;t,s))\|_2^2 \\
	& \quad \leq  4 \pi R (t-s) nt^{n-1} \frac{1}{n!} \int_{[0,t]^n} d\pmb{t} \int_{(\bR^{d})^n} \mu(d\xi_1) \ldots \mu(d\xi_n)
	\bE \left[ e^{i\sum_{j=1}^n B_{t-t_j} \xi_j} \right]
	\ell_R (\sum_{j=1}^n \xi_j).
\end{align*}
Note that
$\bE \left[ e^{i\sum_{j=1}^n B_{t-t_j} \xi_j} \right]
	= \bE \left[ e^{-\mathrm{Var} \sum_{j=1}^n B_{t-t_j} \xi_j} \right]$.
Using relation (3.9) of \cite{NSZ} (with $H_0 = 1/2$), we obtain:
\begin{align}
\label{ineq-I_n-g2-n>1-rough-tight}
	& \|I_n (g_{n,R}^{(2)}(\cdot;t,s))\|_2^2 \leq 4\pi R (t-s) C^n \frac{t^{n(H+1)-1}}{\Gamma(nH+1)}
\end{align}
where $C$ is a constant depends on $T$ and $H$. Thus, by \eqref{ineq-I_1-g2-n=1-rough-tight} and \eqref{ineq-I_n-g2-n>1-rough-tight}, for $0 \le s < t \le T$, and $R>T$, we have
\begin{align}
\label{ineq-I_n-norm-rough-2}
	\sum_{n\geq 1} (p-1)^{n/2}\|I_n (g_{n,R}^{(2)}(\cdot;t,s))\|_2
	\le C R^{1/2} (t-s)^{1/2},
\end{align}
where $C$ is a constant that depends on $T$ and $H$.
Relation \eqref{rough-tight} follows from \eqref{rough-tight-1}, \eqref{ineq-I_n-norm-rough-1} and \eqref{ineq-I_n-norm-rough-2}.

{\em Step 2 (finite dimensional distribution convergence).}
Fix $T>0$. We have to show that for any $m \in \bN_+$, $0\leq t_1<\ldots<t_m\leq T$,
\[
(Q_R(t_1),\ldots,Q_R(t_m)) \stackrel{d}{\to} (\cG(t_1),\ldots,\cG(t_m))
\]

Using the same argument as in the proof of Theorem 1.3.(iii) (Step 2) of \cite{BY}, it is enough to prove that for any $i,j=1,\ldots,m$,
\begin{equation}
\label{var-DF-ij}
{\rm Var}\Big(\langle D F_R(t_i), -DL^{-1} F_R(t_j) \rangle_{\cP_0}\Big) \leq C R.
\end{equation}
By Proposition \ref{propC} (Appendix \ref{appB}), for any $i,j=1,\ldots,m$, we have:
\[
{\rm Var}\Big(\langle D F_R(t_i), -DL^{-1} F_R(t_j) \rangle_{\cP_0}\Big)\leq 2\cA_1+\cA_2,
\]
where
\begin{align*}
\cA_1&=C_H^3 \int_{\bR^6} \|D_{z,y}^2F_{R}(t_i)-D_{z,y'}^2F_{R}(t_i)-D_{z',y}^2F_{R}(t_i)+D_{z',y'}^2F_{R}(t_i)
\|_4\\
& \quad \quad \quad \|D_{w,y}^2F_{R}(t_i)-D_{w,y'}^2F_{R}(t_i)-D_{w',y}^2F_{R}(t_i)+D_{w',y'}^2F_{R}(t_i)
\|_4\\
& \quad  \quad \quad \|D_{w}F_{R}(t_j)-D_{w'}F_{R}(t_j)\|_4\|D_{z}F_{R}(t_j)-D_{z'}F_{R}(t_j)\|_4\\
& \quad  \quad \quad
|y-y'|^{2H-2}|z-z'|^{2H-2}|w-w'|^{2H-2} dydy' dzdz'dwdw',
\end{align*}
and $\cA_2$ is defined by switching the positions of $F_R(t_i)$ and $F_R(t_j)$ in the definition of $\cA_1$.

Similarly to \eqref{A-bound}, it can be proved that $\cA_1 \leq C_T R$ and $\cA_2 \leq C_T R$, where $C_T>0$ is a constant depending on $T$. This proves \eqref{var-DF-ij}.

\appendix

\section{Feyman-Kac formula}
\label{appA}

In this section, we give the Feynman-Kac formula for the moments of the solution to equation \eqref{eq-heat} with time-independent noise, based on the functionals $\cI_{t,s}^{j,k}$ described below. This result is of independent interest. We include it here since the functional $\cI_{t,s}^{1,2}$ appears is needed in Theorem \ref{main}.(i) under Assumption C.

\medskip

For any $\e>0$, let $W^{\e}=\{W^{\e}(\varphi);\varphi \in \cD(\bR)\}$ be the mollified noise, where $W^{\e}(\varphi)=W(\varphi*p_{\e})=\int_{\bR}\varphi(x)\dot{W}^{\e}(x)dx$
and $\dot{W}^{\e}(x)=W(p_{\e}(x-\cdot))$ for any $x \in \bR$.

Consider the equation with noise $W^{\e}$:
\begin{align}
\label{mol-heat}
	\begin{cases}
		\dfrac{\partial u^{\e}}{\partial t} (t,x)
		= \dfrac{1}{2} \dfrac{\partial^2 u^{\e}}{\partial x^2} (t,x) + u^{\e}(t,x) \diamond \dot{W}^{\e}(x), \
		t>0, \ x \in \bR,\\
		u^{\e}(0,x) = 1
	\end{cases}
\end{align}
where $F \diamond W(h)=\delta(Fh)$ denotes the Wick product for any $F \in \bD^{1,2}$ and $h \in \cP_0$.

A process $u^{\e}=\{u_{\e}(t,x);t\geq 0,x\in \bR\}$ is a {\bf mild solution} of \eqref{mol-heat} if
\[
u^{\e}(t,x)=1+\int_0^t \int_{\bR}p_{t-s}(x-y)u^{\e}(s,y)\diamond \dot{W}^{\e}(y)dyds.
\]

Similarly to Proposition 5.2 of \cite{HH09}, it can be proved that the process $u^{\e}$ given by:
\begin{equation}
\label{FK-ue}
u^{\e}(t,x)=\bE^{B} \left[e^{W(A_{t,x}^{\e,B})-\frac{1}{2}\|A_{t,x}^{\e,B}\|_{\cP_0}^2} \right],
\end{equation}
is a solution to \eqref{mol-heat}, where $A_{t,x}^{\e,B}=\int_0^t p_{\e}(x+B_r-y)dr$ and $B=(B_t)_{t\geq 0}$ is a Brownian motion independent of $W$.

The following result is the analogue for the time-independent noise of Proposition 1.3 of \cite{NSZ} which considers a noise with temporal covariance function $\gamma_0(t)=|t|^{2H_0-2}$ for some $H_0 \in (\frac{1}{2},1)$. Our result corresponds formally to the case $H_0=1$.

\begin{proposition} \label{Appen-prop}
Suppose that $\gamma$ is non-negative and non-negative definite and $\mu$ satisfies (D), or Assumption C holds.
For any $t>0$ and $x\in \bR$,
$
\sup_{\e>0}\bE[(u^{\e}(t,x))^n]<\infty$,
the limit $\lim_{\e \downarrow 0}u^{\e}(t,x)$ exists in $L^p(\Omega)$ (for any $p\geq 1$), and coincides with the solution $u(t,x)$ of equation \eqref{eq-heat}. Moreover, for any integer $n\geq 2$,
\[
\bE\big[\prod_{j=1}^n u(t_j,x_j)\big]=\bE\left[\exp\left( \sum_{1\leq j<k\leq n}\cI_{t}^{j,k}(x_j-x_k)\right) \right],
\]
where $\cI_{t,s}^{j,k}(z)$ is defined formally by
\begin{equation}
\label{def-cI}
\cI_{t,t'}^{j,k}(z):=\int_0^t \int_0^{t'} \int_{\bR}e^{-i\xi(B_s^j-B_r^k+z)} \mu(d\xi) drds,
\end{equation}
and is understood as the $L^p(\Omega)$-limit (for any $p\geq 1$) as $\e \downarrow 0$ of
\[
\cI_{t,t',\e}^{j,k}(z)=\int_0^t \int_0^{t'} \int_{\bR}e^{-\e|\xi|^2}e^{-i\xi(B_s^j-B_r^k+z)} \mu(d\xi) drds,
\]
where $B^1,\ldots,B^n$ are i.i.d. $1$-dimensional Brownian motions, independent of $W$.
\end{proposition}

\begin{proof} We only sketch the proof. We omit the details since they are similar to the proof of Proposition 1.3 of \cite{NSZ}.
See also Theorem 3.6 of \cite{HHNT} for the regular noise in space (colored in time), or Theorem 4.2 of \cite{HHLNT} for the rough noise in space (white in time). We consider only the case $t_j=t$ and $x_j=x$ for all $j=1,\ldots,n$. The general case is similar. We let $\cI_{t}^{j,k}=\cI_{t,t}^{j,k}(0)$ and $\cI_{t,\e}^{j,k}=\cI_{t,t,\e}^{j,k}(0)$.

{\em Step 1.} From \eqref{FK-ue}, it follows that
\[
\bE[(u^{\e}(t,x))^n]=\bE\left[\exp\left( \sum_{1\leq j<k\leq n}
\langle A_{t,x}^{\e,B^j}, A_{t,x}^{\e,B^k} \rangle_{\cP_0}\right) \right].
\]
By direct calculation,
$\langle A_{t,x}^{\e,B^j}, A_{t,x}^{\e,B^k} \rangle_{\cP_0}=\cI_{t,\e}^{j,k}$ for any $1\leq j<k\leq n$.

{\em Step 2.} Fix $1\leq j<k\leq n$.
 It can be proved that the limit $\cI_{t}^{j,k}$ of $\cI_{t,\e}^{j,k}$ as $\e \downarrow 0$ exists in $L^p(\Omega)$ (for any $p\geq 1$). Note that
\begin{align*}
\bE[(\cI_{t,\e}^{j,k})^n]&=\int_{\bR^n}e^{-\e\sum_{j=1}^n|\xi_j|^2} \int_{[0,t]^n}
e^{-\frac{1}{2}\sum_{j,k=1}^n(s_j \wedge s_k+r_j\wedge r_k)\xi_j \xi_k}
d\pmb{r_n} d\pmb{s_n} \mu(d\xi_1)\ldots \mu(d\xi_n)\\
&\leq  t^n n! \int_{\bR^n} \int_{T_n(t)} e^{-\sum_{j=1}^n (s_j-s_{j-1})|\xi_j+\ldots+\xi_n|^2}d\pmb{s_n} \mu(d\xi_1)\ldots \mu(d\xi_n).
\end{align*}
From this, one can infer that 
\[
\sup_{\e>0}\bE[e^{\lambda \cI_{t,\e}^{j,k}}]
<\infty.
\]
Hence, the family $\{e^{\sum_{1\leq j<k\leq n}\cI_{t,\e}^{j,k}}\}_{\e>0}$ is uniformly integrable, and
\[
\lim_{\e\downarrow 0}\bE[(u^{\e}(t,x))^n]=\bE[e^{\sum_{1\leq j<k\leq n}\cI_{t,\e}^{j,k})}].
\]
Moreover, $\sup_{\e>0}\bE[(u^{\e}(t,x))^n]<\infty$ and the family $\{u^{\e}(t,x)\}_{\e>0}$ is uniformly integrable.

{\em Step 3.} It follows that $u^{\e}(t,x)$ converges as $\e \to 0$ in $L^p(\Omega)$ (for any $p\geq 2$) to a limit $v(t,x)$. Finally, it can be shown that $v(t,x)$ coincides with $u(t,x)$.
\end{proof}

\begin{remark}
{\rm In the regular case, when $\gamma=\cF \mu$ is a function, it can be proved that
\[
\cI_{t,t'}^{1,2}(z):=\int_0^t \int_0^{t'} \gamma(B_s^1-B_r^2+z) drds.
\]
}

\end{remark}

\section{Second-order Poincar\'e inequality}
\label{appB}

In this section, we state a version of the second-order Poincar\'e inequality for the time-independent noise, rough in space, which is needed for the finite-dimensional convergence in Theorem \ref{main}.(iii), under Assumption C.
This results can be proved using the same argument as in the proof of Proposition 2.4 of \cite{NXZ}, except that here we consider two random variables $F$ and $G$, and there is no time covariance $\gamma_0$. We omit the details.

Suppose that Assumption C holds.
Let $D_{*}^{2,4}$ be the set of random variables $F \in \bD^{2,4}$ such that $DF$ has a measurable modification on $\Omega \times \bR_{+}\times \bR$, $D^2F$ has a measurable modification on $\Omega \times (\bR_+ \times \bR)^2$, $|DF|\in \cH$ and $|D^2F| \in \cH^{\otimes 2}$.

\begin{proposition}
\label{propC}
Let $F,G \in \bD_*^{2,4}$ be such that $\bE(F)=0$ and $\bE(G)=0$. Then
\[
{\rm Var}(\langle DF, -DL^{-1}G\rangle_{\cH})\leq 2\cA_1+\cA_2,
\]
where
\begin{align*}
\cA_1&=C_H^3 \int_{\bR^6} \|D_{z,y}^2F-D_{z,y'}^2F-D_{z',y}^2F+D_{z',y'}^2F\|_4
\|D_{w,y}^2F-D_{w,y'}^2F-D_{w',y}^2F+D_{w',y'}^2F\|_4\\
&\|D_{w}G-D_{w'}G\|_4\|D_{z}G-D_{z'}G\|_4
|y-y'|^{2H-2}|z-z'|^{2H-2}|w-w'|^{2H-2} dydy' dzdz'dwdw',
\end{align*}
and $\cA_2$ is defined by switching the positions of $F$ and $G$ in the definition of $\cA_1$.
\end{proposition}

\end{document}